\documentclass[11pt,reqno,oneside]{amsart}

\usepackage{graphicx, subfigure,}
\usepackage{amssymb}
\usepackage{epstopdf}
\usepackage{amssymb}
\usepackage[a4paper, total={6in, 9.1in}]{geometry}

\usepackage{bm}
\usepackage{amsmath,amsfonts,amsthm,mathrsfs,amssymb,cite}
\usepackage[usenames]{color}

\usepackage{graphics}
\usepackage{framed}

\usepackage{enumerate}
\usepackage{hyperref}
\usepackage{xcolor}
\hypersetup{
	colorlinks,
	linkcolor={blue!80!black},
	urlcolor={blue!80!black},
	citecolor={blue!80!black}
}
\usepackage[capitalize,noabbrev]{cleveref}

\usepackage{xcolor}

\numberwithin{equation}{section}

\newtheorem{theorem}{Theorem}[section]
\newtheorem{lemma}[theorem]{Lemma}

\theoremstyle{definition}

\theoremstyle{remark}
\newtheorem{remark}[theorem]{Remark}

\numberwithin{equation}{section}

\newcounter{saveeqn}
\newcommand{\ba}{\begin{array}}
	\newcommand{\ea}{\end{array}}

\newcommand{\bea}{\begin{eqnarray*}}
	\newcommand{\eea}{\end{eqnarray*}}
\newcommand{\bean}{\begin{eqnarray}}
	\newcommand{\eean}{\end{eqnarray}}

\allowdisplaybreaks[4]

\usepackage{xstring} % 

\newcommand{\AuthorInfo}[4]{%
	\textsc{#1}%
	\IfStrEq{#4}{true}{$^{*}$}{}\\
	#2\\
	\textit{E-mail:} \texttt{#3}%
	\IfStrEq{#4}{true}{\\\textit{$^{*}$Corresponding author.}}{}
	\par\vspace{0.5em}
}

\usepackage{authblk}

\title[A New Quasi-Singularity Formation Mechanism]
{A New Quasi-Singularity Formation Mechanism for Second-order Hyperbolic Equations}

\date{} % Activate to display a given date or no date (if empty),
\begin{document}
	
	\maketitle
	
	\begin{center}
		\bigskip
		\footnotesize

		\AuthorInfo{Huaian Diao}{School of Mathematics and Key Laboratory of Symbolic Computation and Knowledge Engineering of Ministry of Education, Jilin University, Changchun 130012, People's Republic of China}{diao@jlu.edu.cn, hadiao@gmail.com}{false}
		
		\AuthorInfo{Xieling Fan}{School of Mathematics, Central South University, Changsha, China and Department of Mathematics, City University of Hong Kong, Kowloon, Hong Kong SAR, China}{fanxieling@outlook.com, xielinfan2-c@my.cityu.edu.hk}{false}
		
		\AuthorInfo{Hongyu Liu}{Department of Mathematics, City University of Hong Kong, Kowloon, Hong Kong SAR, People's Republic of China}{hongyu.liuip@gmail.com, hongyliu@cityu.edu.hk}{false}
	\end{center}

	\normalsize

	\begin{abstract}
		This paper investigates a novel mechanism for quasi-singularity formation in both linear and nonlinear hyperbolic wave equations in two and three dimensions. We prove that over any finite time interval, there exist inputs such that the H\"older norm of the resulting wave field exceeds any prescribed bound. Conversely, the set of such almost-blowup points has vanishing measure when the aforementioned bound goes to infinity. This phenomenon thus defines a quasi-singular state, intermediate between classical singularity and regularity. Crucially, both the equation coefficients and the inputs can be arbitrarily smooth; the quasi-singularity arises intrinsically from the structure of the hyperbolic wave equation combined with specific input characteristics.

		\medskip

		\noindent{\bf Keywords:}~~second-order hyperbolic equations; quasi-singularity formation; Herglotz approximation; interior transmission problem; Galerkin method; Banach fixed point theorem.
		\noindent{\bf 2020 Mathematics Subject Classification:}~~35L71, 35L67, 35L15, 35B44
		
	\end{abstract}

	\section{Introduction}
	
	\subsection{Mathematical setup and summary of major results}\label{setup}

	Initially we focus on the mathematical models for both linear and nonlinear second-order hyperbolic equations in two and three spatial dimensions in this paper, where the mechanism for these mathematical models will be elaborated in more detail later. Let the unknown function $u:\Omega \times [0,T] \rightarrow \mathbb{C}$ satisfy the homogeneous initial-boundary value problem for a linear second-order hyperbolic equation:
	\begin{equation}\label{1.1}
		\begin{cases}
			L_1u = 0 & \text{in } \Omega \times (0,T], \\
			u(\mathbf{x},0) = \phi_1(\mathbf{x}), \quad \partial_t u(\mathbf{x},0) = \phi_2(\mathbf{x}) & \text{in } \Omega, \\
			u(\mathbf{x},t) = \psi(\mathbf{x},t) & \text{on } \partial\Omega \times [0,T],
		\end{cases}
	\end{equation}
	and let $U: \mathbb{R}^d \times [0,T] \rightarrow \mathbb{C}$ satisfy the Cauchy problem for a nonlinear  second-order hyperbolic equation:
	\begin{equation}\label{1.2}
		\left\{
		\begin{aligned}
			&L_2U + N(U,\partial_{t}U,\nabla U) = 0 
			&& \text{in } \mathbb{R}^{d} \times (0,T], \\
			&U(\mathbf{x},0) = \Phi_1(\mathbf{x}), \quad \partial_t U(\mathbf{x},0) = \Phi_2(\mathbf{x}) 
			&& \text{in } \mathbb{R}^d,
		\end{aligned}
		\right.
	\end{equation}
	where the following conditions hold:
	\begin{enumerate}
		\item $\Omega$ is a bounded Lipschitz domain in $\mathbb{R}^d$, $d=2, 3$, and $0<T<\infty$.
		
		\item $L_1=\partial_{tt}^2  - \nabla \cdot (\mathbf{A}_1\nabla ) + \mathbf{b}_1\cdot \nabla  + c_1$, $L_2=\partial_{tt}^2  - \nabla \cdot (\mathbf{A}_2\nabla ) + \mathbf{b}_2\cdot \nabla  + c_2$.
		
		\item The complex-valued coefficients $\mathbf{A}_1 = (a_1^{ij}(\mathbf{x}))_{i,j=1}^d$, $\mathbf{b}_1= (b_1^i(\mathbf{x},t))_{i=1}^d$, $c_1(\mathbf{x},t)$ satisfy:
		\begin{enumerate}
			\item $\overline{\mathbf{A}}_1^{\top} \! (\mathbf{x}) = \mathbf{A}_1(\mathbf{x}), ~\forall \mathbf{x} \in \Omega$;
			
			\item $\overline{\xi}^{\top} \mathbf{A}_1(\mathbf{x}) \xi \geq \theta \lvert \xi \rvert^{2}$ for a.e. $\mathbf{x} \in \Omega$, $\forall \xi \in \mathbb{C}^d$, with constant $\theta > 0$;
			
			\item $\mathbf{A}_1\in W^{2,\infty}(\Omega)$;
			
			\item $\mathbf{b}_1, c_1 \in W^{1,\infty}\big(0,T; W^{1,\infty}(\Omega)\big) \cap W^{2,\infty}\big(0,T; L^{\infty}(\Omega)\big)$,
		\end{enumerate}
		where $\overline{\cdot}$ denotes the complex conjugate and $\top$ denotes the transpose.
		
		\item $D$ is a bounded Lipschitz domain such that $D\Subset\Omega$, $\Omega \backslash \overline{D}$ is connected and  $\mathbf{A}_1$, $\mathbf{b}_1$, $c_1$ have compact support in $D$, i.e.,
		\begin{equation*}\label{compact support}
			\operatorname{supp}(\mathbf{A}_1 - \mathbf{I}) \subset D,\quad
			\operatorname{supp} \mathbf{b}_1(\cdot,t) \subset D,\quad
			\operatorname{supp} c_1(\cdot,t) \subset D \quad \forall t \in [0,T].
		\end{equation*}
		
		\item The complex-valued coefficients $\mathbf{A}_2 = (a_2^{ij}(\mathbf{x},t))_{i,j=1}^d$, $\mathbf{b}_2= (b_2^i(\mathbf{x},t))_{i=1}^d$, $c_2(\mathbf{x},t)$ satisfy:
		\begin{enumerate}
			\item $\overline{\mathbf{A}}_2^{\top} \! (\mathbf{x},t) = \mathbf{A}_2(\mathbf{x},t), ~\forall (\mathbf{x},t) \in \mathbb{R}^d\times [0,T]$;
			
			\item $\overline{\xi}^{\top} \mathbf{A}_2(\mathbf{x},t) \xi \geq \theta \lvert \xi \rvert^{2}$ for all $(\mathbf{x},t) \in \mathbb{R}^d \times [0,T]$, $\forall \xi \in \mathbb{C}^d$, with constant $\theta > 0$;
			
			\item $\mathbf{A}_2\in {W^{1,\infty}\left(0,T; W^{3,\infty}(\mathbb{R}^d)\right)}$;
			
			\item $\mathbf{b}_2, c_2 \in L^{\infty}\big(0,T; W^{2,\infty}(\mathbb{R}^d)\big)$.
		\end{enumerate}
		
		\item $\mathbf{A}_2$, $\mathbf{b}_2$, $c_2$ have compact support in $D$, i.e.,
		\begin{equation*}\label{compact support2}
			\operatorname{supp}(\mathbf{A}_2(\cdot,t) - \mathbf{I}) \subset D,\quad
			\operatorname{supp} \mathbf{b}_2(\cdot,t) \subset D,\quad
			\operatorname{supp} c_2(\cdot,t) \subset D \quad \forall t \in [0,T].
		\end{equation*}
		
		\item The nonlinearity $N: \mathbb{C} \times  \mathbb{C} \times  \mathbb{C}^{d}\to \mathbb{C}$ is a complex polynomial given by
		\begin{equation}
			N(U,\partial_{t} U,\nabla U) =\sum_{k=2}^{l_{0}} \left( \alpha_k U^k + \beta_k ({\partial_t U})^k + \sum_{j=1}^d \gamma_{k,j} (\partial_{x_j} U)^k \right),
			\label{eq:nonlinearity}
		\end{equation}
		where $l_0 \geq 2$ is a fixed integer.
		
		\item The coefficients in nonlinearity $N$ satisfy$:$
		\[
		\alpha_k, \beta_k, \gamma_{k,j} \in C^{\infty}([0,\infty)\times \mathbb{R}^d;\mathbb{C}), \quad k=2,\dots,l_0,\ j=1,\dots,d,
		\]
		with compact support in $D$, i.e.,
		\[
		\operatorname{supp}(\alpha_k(\cdot,t)),\ \operatorname{supp}(\beta_k(\cdot,t)),\ \operatorname{supp}(\gamma_{k,j}(\cdot,t)) \subset D, \quad \forall t\in[0,\infty).
		\]
		
		\item The initial data $\phi_1,\phi_2\in L^2{(\Omega)}$, $\Phi_1,\Phi_2 \in L^{2}_{\mathrm{loc}}(\mathbb{R}^d)$, and boundary data $\psi \in L^{2}(\partial \Omega)$ will be specified in a later section.
		
		\item $\omega$ is a fixed positive constant, which denotes the wave number in \eqref{wave number} and the transmission eigenvalue in \eqref{Interior Transmission Problem}.
	\end{enumerate}
	
	%In the context of wave physics, the coefficients of the above operators $ L_i $ ($ i = 1, 2 $) characterize the medium through which the wave propagates. The tensor $ \mathbf{A}_i $ describes the medium's inherent capacity for wave transmission and its directional dependence (anisotropy). The vector field $ \mathbf{b}_i $ represents the advective velocity of the medium, which transports the wave. The scalar coefficient $ c_i $ accounts for energy dissipation and shifts in the wave's natural frequency. In acoustics, for a homogeneous, isotropic fluid, the tensor $ \mathbf{A}_i $ simplifies to $ c^2 \mathbf{I} $, where $ c $ denotes the speed of sound in the medium.
	
	In this paper, we construct a novel mathematical mechanism governing {\it quasi-singularity} formation for both linear and nonlinear second-order hyperbolic equations in two and three spatial dimensions. The quasi-singularity formation is characterized by two key features:
	
	\begin{enumerate}
		\renewcommand{\labelenumi}{(\roman{enumi})}
		\item For any prescribed distinct $n$ points $\mathbf{x}_1,\mathbf{x}_2,\ldots,\mathbf{x}_n \in \partial D$, where $D$ is the support of the anisotropic terms $\mathbf A_1$ in \eqref{1.1} or $\mathbf A_2$ in \eqref{1.2}, the first-order coefficients $\mathbf b_i$ ($i=1,2$) and zero-order terms $c_i$ ($i=1,2$) in \eqref{1.1} or \eqref{1.2}, and the nonlinear term $N$ in \eqref{1.2}, we rigorously prove---by selecting suitably tuned smooth initial-boundary or initial input---that there exist almost-blowup regions near each $\mathbf{x}_i$ ($i=1,\ldots,n$). In these regions, the $L^\infty$-norms of the solutions $u$ to \eqref{1.1} and $U$ to \eqref{1.2} exceed any prescribed large positive constant $\mathcal M$ for almost every $t \in [0,T]$, as stated in Theorems \ref{thm:main} and \ref{thm:main2}. Here $T\in \mathbb R_+$ can be arbitrarily chosen. 
		\item We provide precise characterizations of the measure of these almost-blowup regions in terms of $\mathcal M$, which quantifies the almost-blowup behavior of the solutions $u$ and $U$ to \eqref{1.1} and \eqref{1.2} under the tuned input. Our detailed analysis reveals that as $\mathcal M \to \infty$, the measure of the almost-blowup regions tends to zero. We thus term this almost-blowup phenomenon, marked by the vanishing measure of the almost-blowup regions, ``quasi-singularity formation''.
	\end{enumerate}
	
	We are now in a position to state our main theorems. For initial-boundary value problem of a linear second-order hyperbolic equation \eqref{1.1}, we have:
	\begin{theorem}\label{thm:main}
		For any arbitrarily large $\mathcal{M} > 0$, $n$ distinct points $\mathbf{x}_1,\mathbf{x}_2,\ldots,\mathbf{x}_n \in \partial D$, and any time interval $[0,T]$ with $0 < T < \infty$, by prescribing appropriate initial and boundary inputs $($as specified in \eqref{initial and boundary condition}$)$, there exists a small parameter
		$$ r_0 = \min\left\{\,
		\frac{1}{3}(\mathcal{M} + 1)^{-1},\ 
		\frac{1}{6}\min_{1 \le i < j \le n} |\mathbf{x}_i - \mathbf{x}_j|,\ 
		\frac{1}{3}\operatorname{dist}(D, \partial \Omega)
		\,\right\}$$
		such that problem \eqref{1.1} admits a unique solution $u$ satisfying $:$
		$$
		u \in L^{\infty}\big(0,T;H^{1}(\Omega)\big), \quad 
		u_t \in L^{\infty}\big(0,T;L^{2}(\Omega)\big),
		$$
		with the following properties $:$
		\begin{enumerate}
			\item Improved H\"older regularity in $\Omega\backslash\overline{D}:$
			for any bounded Lipschitz domain $\Omega' \Subset \Omega\backslash\overline{D}$,
			\[
			u \in L^{\infty}\big([0,T]; C^{1,\frac{1}{2}}(\overline{\Omega}')\big).
			\]
			\item Simultaneous spatial gradient almost-blowup near $\mathbf{x}_1,\mathbf{x}_2,\dots,\mathbf{x}_n :$ for almost every $t \in [0,T]$,
			\[
			\|\nabla u(\cdot,t)\|_{L^{\infty} (B_{3r_0}(\mathbf{x}_i) \backslash \overline{D})} \geq \mathcal{M}, \quad i=1,2,\dots,n.
			\]
			\item Vanishing measure of the almost-blowup set $:$ 
			for almost every $t \in [0,T]$,
			\begin{equation*}
				\begin{aligned}
					&\operatorname{meas}\left\{ \mathbf{x} \in \bigcup_{i=1}^{n} \left(B_{3r_0}(\mathbf{x}_i)\setminus \overline{D}\right) : |\nabla u(\mathbf{x},t)| > \mathcal{M} \right\}\\
					&	< \begin{cases}
						36n\pi r_0^3 < \dfrac{4n\pi}{3(\mathcal{M} + 1)^3} \to 0,&d=3, \\
						9n\pi r_0^2 < \dfrac{n\pi}{(\mathcal{M}+1)^2} \to 0,&d=2,
					\end{cases}
				\end{aligned}
			\end{equation*}
			as $\mathcal{M} \to \infty$.
		\end{enumerate}
	\end{theorem}
	
	Similarly, for Cauchy problem of a nonlinear second-order hyperbolic equation \eqref{1.2}, we have:
	\begin{theorem}\label{thm:main2}
		For any arbitrarily large $\mathcal{M} > 0$, $n$ distinct points $\mathbf{x}_1,\mathbf{x}_2,\ldots,\mathbf{x}_n \in \partial D$, and any time interval $[0,T]$ with $0 < T < \infty$, by prescribing appropriate initial inputs $($as specified in \eqref{initial condition}$)$, there exists a small parameter
		$$
		r_0  = \min\left\{\frac{1}{3}(\mathcal{M} + 1)^{-1},\  \frac{1}{6}\min_{1 \leq i < j \leq n} |\mathbf{x}_i - \mathbf{x}_j|\right\}
		$$
		such that problem \eqref{1.2} admits a unique solution $U$ satisfying:
		$$
		U \in L^{\infty}\big(0,T;H^{3}_{\mathrm{loc}}(\mathbb{R}^d)\big), \quad 
		U_t \in L^{\infty}\big(0,T;H^{2}_{\mathrm{loc}}(\mathbb{R}^d)\big),
		$$
		with the following properties$:$
		\begin{enumerate}
			\item \textit{H\"older regularity}: For any bounded Lipschitz domain $\Omega' \subset \mathbb{R}^d$,
			$$
			U \in L^{\infty}\big([0,T]; C^{1,\frac{1}{2}}(\overline{\Omega}')\big).
			$$
			
			\item Simultaneous spatial gradient almost-blowup near $\mathbf{x}_1,\mathbf{x}_2,\dots,\mathbf{x}_n:$ for almost every $t \in [0,T]$,
			$$
			\|\nabla U(\cdot,t)\|_{L^{\infty} (B_{3r_0}(\mathbf{x}_i) \backslash \overline{D})} \geq \mathcal{M}, \quad i=1,2,\dots,n.
			$$
			
			\item Vanishing measure of the almost-blowup set: 
			for almost every $t \in [0,T]$,
			\begin{equation*}
				\begin{aligned}
					&\operatorname{meas}\left\{ \mathbf{x} \in \bigcup_{i=1}^{n} \left(B_{3r_0}(\mathbf{x}_i)\setminus \overline{D}\right) : |\nabla U(\mathbf{x},t)| > \mathcal{M} \right\}\\
					&	< \begin{cases}
						36n\pi r_0^3 < \dfrac{4n\pi}{3(\mathcal{M} + 1)^3} \to 0,&d=3, \\
						9n\pi r_0^2 < \dfrac{n\pi}{(\mathcal{M}+1)^2} \to 0,&d=2,
					\end{cases}
				\end{aligned}
			\end{equation*}
			as $\mathcal{M} \to \infty$.
		\end{enumerate}
	\end{theorem}
	
	\begin{remark}
		We refer to the solutions $u$ and $U$ in the above problems as {\it almost-blowup solutions}. The associated {\it almost-blowup points} are those prescribed \(\mathbf{x}_1, \mathbf{x}_2, \dots, \mathbf{x}_n\)  where the solution's spatial gradient exceeds the prescribed bound \(\mathcal{M}\).
	\end{remark}
	
	\begin{remark}
		The novel mechanism for quasi-singularity formation in second-order hyperbolic equations exhibits the following key features. 
		
		Firstly, it applies to a broad class of linear and nonlinear second-order hyperbolic equations in two and three dimensions. The equation coefficients inside the region \(D\) can be quite general (e.g., smooth coefficients or nonlinear terms given by arbitrary polynomials), while outside \(D\) the equation reduces to the classical wave equation.
		
		Secondly, for any terminal time $T\in \mathbb R_+$, any {arbitrarily large} prescribed bound \(\mathcal{M}\) and any  {arbitrary} set of \(n\) prescribed distinct points on \(\partial D\), one can construct suitably designed smooth initial and/or boundary inputs such that the magnitude of the spatial gradient of the solutions to problems \eqref{1.1} and \eqref{1.2} exceeds $\mathcal{M}$ in a neighborhood of each prescribed point \(\mathbf{x}_1, \mathbf{x}_2, \dots, \mathbf{x}_n\) for almost every $t \in [0, T]$, in both linear and nonlinear settings. Here $T \in \mathbb R_+$  can be chosen arbitrarily. 
		
		Thirdly, the constructed almost-blowup points are confined within the balls $B_{3r_0}(\mathbf{x}_i)$. Consequently, as $\mathcal{M} \to \infty$, the measure of the almost-blowup set tends to zero, which is consistent with the well-posedness of the solutions. This demonstrates a fundamental trade-off: attaining a higher gradient requires reducing the spatial extent of the quasi-singularity region, while a more moderate gradient permits the almost-blowup region to occupy a larger area. Furthermore, when the prescribed points \(\mathbf{x}_1, \mathbf{x}_2, \dots, \mathbf{x}_n\) are closely spaced, their corresponding almost-blowup regions must be sufficiently constrained to prevent mutual interaction.
	\end{remark}
	
	To the best of our knowledge, this work is the first to establish such quasi-singularity formation and the existence of almost-blowup solutions for second-order hyperbolic equations in both linear and nonlinear settings. For a detailed discussion, see Subsection~\ref{subsection 1.3}. We provide several remarks to discuss the methodology for establishing the main results, comparisons between the assumptions and conclusions of Theorems~\ref{thm:main} and~\ref{thm:main2}, and comments regarding the a priori assumption on the geometry for generating quasi-singularity formation.
	
	\begin{remark}
		The quasi-singularity formation mechanism we construct for second-order hyperbolic equations is grounded in spectral analysis, specifically involving transmission eigenfunctions of the interior transmission problem \cite{ColtonKress} and their approximation by Herglotz wave functions. Using these tools, we carefully construct a key function \( u_0 \), which plays a central role in designing the initial and boundary inputs for both linear and nonlinear second-order hyperbolic equations. Further details are provided in Section \ref{section 3}.
	\end{remark}
	
	\begin{remark}
		The solution to nonlinear problem \eqref{1.2} exhibits higher regularity than that of the linear problem \eqref{1.1}, owing to the improved spatial regularity of its coefficients. In the linear case, to compensate for the limited spatial regularity, we employ interior elliptic regularity for \( u(t) \) in \( \Omega \setminus \overline{D} \) for almost every \( t \in [0,T] \), and further establish that \( u \in L^{\infty}\left(0,T; H^{3}_{\mathrm{loc}}(\Omega \setminus \overline{D})\right) \). Therefore, in both settings, by Sobolev embedding theorem the solutions maintain \( C^{1,\frac{1}{2}} \) regularity near the almost-blowup points for almost every $t \in [0,T] $, which ensures that their spatial gradients exist in the classical sense. Moreover, if the time-regularity of the coefficients is strengthened, the almost everywhere in time result can be improved to hold for all \( t \in [0,T] \).
	\end{remark}
	
	\begin{remark}
		In this work, the set of almost-blowup points is characterized using balls (in $3D$) or disks (in $2D$). We believe that employing other types of sets to describe such quasi-singularities is also feasible, though it would require more sophisticated technical approaches.
	\end{remark}

	\begin{remark}
		For both the linear problem \eqref{1.1} and the nonlinear problem \eqref{1.2}, the terminal time $T \in \mathbb{R}_+$ of the interval $[0, T]$ can be chosen arbitrarily large while still guaranteeing the almost blow-up behaviors of the solutions. In the linear case, this follows directly from the classical well-posedness theory for linear second-order hyperbolic equations. In the nonlinear case, however, such long-time existence requires a more refined analysis. By carefully tuning the initial data and combining spectral analysis with an application of the Banach fixed-point theorem, we are able to extend the solution to \eqref{1.2} over any prescribed time interval $ T $. For further details, see Remark~\ref{remark 5.3} and the proof of Theorem~\ref{thm:main2}.
	\end{remark}

	\begin{remark}
		Setting the polynomial nonlinearity $ N = 0 $ in problem \eqref{1.2} reduces it to the Cauchy problem for a linear second-order hyperbolic equation, which is still amenable to our quasi-singularity formation mechanism. Note that even in this linear case ($ N = 0 $ in \eqref{1.2}), the problem cannot be transformed into \eqref{1.1}, because \eqref{1.2} is defined on $ \mathbb{R}^d \times (0,T] $ with an anisotropic coefficient matrix $ \mathbf{A}_2(\mathbf{x}, t) $, whereas \eqref{1.1} is an initial-boundary value problem on a bounded domain with an anisotropic coefficient matrix $ \mathbf{A}_1(\mathbf{x}) $.

	\end{remark}

	\subsection{Physical motivation and existing studies}
	Understanding wave phenomena in heterogeneous media is essential for addressing numerous challenges in applied mathematics and physics. Wave propagation in heterogeneous media provides a rich physical context for the second-order hyperbolic equations in \eqref{1.1} and \eqref{1.2}, where material properties change abruptly across interfaces. These equations encompass seismic wave propagation in layered geological structures (e.g., oil reservoirs with embedded fractures) where $\mathbf{A}_1$ encodes spatially varying elasticity tensors \cite{fichtner2011full}, electromagnetic wave scattering in photonic crystals with compactly supported defects governed by $\mathbf{A}_2$ representing permittivity/permeability tensors \cite{joannopoulos2008photonic}, and acoustic wave transmission through composite materials with smooth inclusions captured by the compact-support coefficients in $D$ \cite{craster2012acoustic}. These diverse applications share the critical feature that wave dynamics remain classical (i.e., $\mathbf{A}_1,\mathbf{A}_2 = \mathbf{I}$) outside the bounded domain $D$, while exhibiting complex behavior within $D$. These important applications in physics motivate us to investigate the mathematical theory hidden behind them.
	
	Hyperbolic equations, particularly wave equations, hold a fundamental position in the theory of partial differential equations. A central theme in their study is finite-time blow-up and singularity formation. While solutions to the classical linear wave equation ($\partial_{tt}u - \Delta u = 0$) do not exhibit blow-up preserving the regularity of the initial data \cite{Evans} and propagating singularities along characteristics without amplification \cite{Hormander} the situation changes dramatically in the nonlinear setting.
	
	Nonlinear wave equations may fail to admit global solutions, as singularities can form in finite time through various blow-up mechanisms. The occurrence of such blow-up depends critically on the structure of the equation, the spatial dimension, and the nature of the nonlinearity. In what follows, we review several classical blow-up results, starting with foundational approaches that underpin the modern analysis of these phenomena.
	
	In \cite{glassey1973blowup}, semilinear wave equations of the form $\partial_{tt}u-\Delta u=f(u)$ are studied. Under the assumption that $f$ is a positive convex function, finite-time blow-up is established for both bounded domains and the Cauchy problem. The argument relies on an ODE-type blow-up criterion: by deriving a differential inequality for an integral functional of the solution, the PDE dynamics are reduced to a nonlinear ordinary differential equation that diverges in finite time. In \cite{levine1974}, an abstract blow-up framework was formulated in a Hilbert space setting for nonlinear wave equations of the form $Pu_{tt}=-Au+\mathcal{F}(u)$. The method employed there is based on a concavity argument, showing that a suitably chosen energy-related functional becomes strictly concave in time, which forces finite-time blow-up under certain sign and coercivity conditions.
	
	A landmark contribution to the field was made by John \cite{John} in the analysis of the prototypical equation  
	$$
	\partial_{tt}u-\Delta u =|u|^p
	$$  
	in three spatial dimensions. It was established that for $1<p<1+\sqrt{2}$ now known as the John exponent any nontrivial smooth initial data with compact support lead to finite-time blow-up. The underlying mechanism reveals a key insight: for these subcritical exponents, the defocusing nonlinearity dominates the dispersive effects of the wave operator, causing energy to concentrate along the characteristic cone. A central element of the proof is to show that if a solution satisfies the inequality $\partial_{tt}u-\Delta u \gtrsim |u|^p$, then it must remain entirely supported within the backward light cone $|x-x_0| \leq t_0 - t$. As this cone contracts to a point in finite time, the confinement of the solution's support forces a concentration of energy, ultimately leading to the solution becoming singular in finite time, regardless of the initial data's size. This foundational result was later extended to higher dimensions by Sideris\cite{Sideris}, who further clarified how the critical exponent for such blow-up phenomena depends on the spatial dimension.
	
	In the quasilinear setting, the blow-up mechanisms become more intricate. A celebrated result of Alinhac \cite{alinhac1999} established blow-up for small-data solutions of a certain two-dimensional quasilinear wave equation,
	$$
	\partial_{tt}u-\Delta u=u_t u_{tt}.
	$$
	Not only was finite-time blow-up proved, but a sharp asymptotic description of the solution's lifespan and profile near the blow-up point was also obtained. The lifespans theorem of \cite{alinhac1999} demonstrates that, despite having smooth initial data, the solution as it approaches the maximal existence time retains only $C^1$ regularity globally, indicating a catastrophic loss of regularity. The underlying blow-up mechanism is of the so-called “cusp-type” (for further details, see \cite{alinhac1995blowup}). This work illustrates how geometric considerations grow in importance with increasing equation complexity, thereby paving the way for more sophisticated analyses of singularity formation.
	
	Recent decades have witnessed increasingly refined techniques for understanding blow-up in critical and supercritical settings, where geometric and analytic approaches converge to reveal deeper structures. A notable geometric development comes from the work of Rodnianski and Sterbenz \cite{rodnianski2010formation}, which studied singularity formation in the critical $O(3)$ $\sigma$-model wave maps from $\mathbb{R}^{2+1}$ to $\mathbb{S}^2$. Their analysis rigorously constructed solutions that blow up in finite time through self-focusing (shrinking) of a corresponding harmonic map. The underlying mechanism involves the concentration of a rescaled harmonic map profile whose scale tends to zero as the singularity forms. This result provided a geometric model of blow-up and connected singularity formation in wave equations to the dynamics of harmonic maps. Another important direction is the renormalization method. In \cite{krieger2009slow}, Krieger, Schlag, and Tataru constructed slow blow-up solutions for the $H^1(\mathbb{R}^3)$-critical focusing semilinear wave equation  
	\begin{equation}\label{eq 1.4}
		\partial_{tt}u-\Delta u=u^5.
	\end{equation}  
	In a related work \cite{Krieger}, the same authors analyzed critical wave maps and also constructed blow-up solutions. The key mechanism relies on a renormalization procedure applied to the soliton profile $W(r)=\left(1+r^2/3\right)^{-\frac{1}{2}}$, which is the stationary solution of \eqref{eq 1.4}. For further blow-up mechanisms in nonlinear wave equations, we refer to \cite{collot2018typeII,duyckaerts2013classification,duyckaerts2011universality,kato1980,hormander1997lectures,kenig2008global,MerleZaag,Sogge,Strauss,DonningerRao2020Blowup,DonningerSchorkhuber2016} and references therein.
	
	\subsection{Main contributions and technical discussions}\label{subsection 1.3}

	Unlike the classical singularity formation mechanisms for nonlinear wave equation described above, our work establishes a novel mechanism  quasi-singularity formation in a broad class of anisotropic wave equations defined by their behavior inside and outside a bounded domain $D$. Specifically, we consider anisotropic wave equations that reduce to the classical wave equation outside $D$, while inside $D$ they governed by variable-coefficient operators and may include arbitrary polynomial nonlinearities. This mechanism represents a distinct behavior that differs fundamentally from classical blowup in both its qualitative nature and mathematical structure. Specifically, quasi-singularity formation exhibits the following precise characteristics:
	\begin{enumerate}
		\renewcommand{\labelenumi}{(\roman{enumi})}
		\item %\textbf{Controlled almost-blowup behavior}: 
		For any prescribed $\mathcal{M} > 0$ and any prescribed number $n$ of distinct points $\mathbf{x}_1, \dots, \mathbf{x}_n$ on $\partial D$, we construct inputs (specifically, initial and/or boundary input) that yield solutions $u$ and $U$ to problems \eqref{1.1} and \eqref{1.2}, respectively. In the specified regions near these points $\mathbf{x}_1, \dots, \mathbf{x}_n$, the $L^\infty$-norms of $\nabla u$ and $\nabla U$ exhibit almost-blowup behavior bounded below by $\mathcal{M}$ for almost every time $t$ in the interval $[0, T]$, where $T$ can be arbitrarily chosen in advance. For the nonlinear problem \eqref{1.2}, the solution exists for an arbitrarily large time interval. This long-time existence is achieved through a combination of spectral analysis, application of the Banach fixed-point theorem, and careful tuning of the initial data. For further details, see Remark~\ref{remark 5.3} and the proof of Theorem~\ref{thm:main2}.
		
		\item 
		%\textbf{Vanishing measure of the almost-blowup set}: 
		
		Due to the well-posedness of problems \eqref{1.1} and \eqref{1.2}, the measure of the almost-blowup regions---where the almost-blowup behavior is bounded below by $\mathcal{M}$---vanishes for almost every $t \in [0,T]$ as $\mathcal{M} \to \infty$. We provide intricate characterizations of the measure of these almost-blowup regions in Theorems~\ref{thm:main} and~\ref{thm:main2}. In this sense, the mechanism proposed in this paper is termed the \emph{quasi-singularity} approach: it induces almost-blowup behavior in the solutions $u$ and $U$ to \eqref{1.1} and \eqref{1.2} by selecting tuned inputs in specified regions near any given distinct points on the boundary of the support, which characterizes anisotropy or nonlinearity in wave propagation. The vanishing properties of these almost-blowup regions are rigorously analyzed therein. Stronger almost-blowup behavior in $u$ and $U$ corresponds to a smaller measure of the almost-blowup regions. This quantitative decay property of the measure with respect to $\mathcal{M}$ stands in contrast to classical blowup, where the singular set typically has positive measure at the blowup time. Indeed, classical blow-up studies rely on certain quantities that compromise the well-posedness of the problem by disrupting the regularity of the underlying solution and its lifespan. In contrast, in this paper, under the specified smooth and finely tuned inputs--even with the weaker regularity coefficients as specified in Items 3 and 5 of \eqref{1.1} or \eqref{1.2}, whose supports lie in a bounded Lipschitz domain--the almost-blowup solutions $u$ and $U$ to problems \eqref{1.1} and \eqref{1.2}, respectively, retain spatial $H^1$ and $H^3$ regularity for almost every $t \in [0,T]$. Moreover, within the almost-blow-up regions, the solutions maintain local spatial $C^{1,1/2}$ regularity for almost every $t \in [0,T]$, where $T > 0$ can be chosen arbitrarily.
	\end{enumerate}

	This quasi-singularity mechanism arises from the interplay between the spectral theory of transmission eigenfunctions for the interior transmission problem~\cite{ColtonKress} and Herglotz wave approximations, which uncovers a novel pathway for inducing arbitrary wave amplification---in terms of energy---near the pre-specified almost-blowup regions, with respect to the almost-blowup quantity $\mathcal{M}$. Leveraging Herglotz wave approximations and the well-posedness of \eqref{1.1} and \eqref{1.2}, we construct the requisite initial and/or boundary inputs for both the linear problem \eqref{1.1} and the nonlinear problem \eqref{1.2}; see \eqref{initial and boundary condition} and \eqref{initial condition} for details. With these inputs, we analyze the well-posedness of the associated auxiliary linear problem \eqref{remaining term} and nonlinear auxiliary problem \eqref{nonlinear hyperbolic equation with source term and zero initial condition} in Sections~\ref{sec:4} and~\ref{sec:5}, respectively. The solution to \eqref{remaining term} exhibits interior H\"older regularity, with estimates for the corresponding H\"older norm with respect to the accuracy of the Herglotz wave approximation established in Theorem~\ref{3.1}. We employ Galerkin method and elliptic interior regularity theory to prove Theorem~\ref{3.1}.  Likewise, for the nonlinear auxiliary problem \eqref{nonlinear hyperbolic equation with source term and zero initial condition}, Theorem~\ref{thm:5.2} establishes global H\"older regularity, along with lifespan estimates and bounds on the corresponding H\"older norm with respect  to the accuracy from Theorem~\ref{3.1}. Using Banach fixed theorem and energy estimates with suitable spaces, we can prove Theorem~\ref{thm:5.2}.  Finally, by integrating these results and strategically tuning the key parameters identified in Theorem~\ref{thm:main_result}, we establish the quasi-singularity formation and almost-blowup solutions for both problems \eqref{1.1} and \eqref{1.2}.

	The remainder of this paper is organized as follows. Section 2 establishes the requisite mathematical preliminaries. In Section 3, we develop a unified framework for constructing initial and boundary inputs for both problems \eqref{1.1} and \eqref{1.2}. Section 4 presents the complete proof of Theorem \ref{thm:main_result}, establishing well-posedness for the linear auxiliary problem \eqref{remaining term} and constructing almost-blowup solutions through strategic parameter tuning as quantified in the theorem. Finally, Section 5 extends this methodology to the nonlinear setting in Theorem \ref{thm:main2}, where a fixed-point argument establishes well-posedness for the auxiliary problem \eqref{nonlinear hyperbolic equation with source term and zero initial condition} and proves the analogous quasi-singular solution existence.

	\section{Preliminaries}
	
	Throughout this paper, the constant $C$ may change from line to line. In the statements of theorems, we explicitly specify the parameters on which constants depend. However, within proofs, for brevity and to avoid cumbersome notation, we typically do not detail the precise dependence of constants on various parameters. For particularly significant constants that play a crucial role in the argument, we assign specific labels such as $C_1$, $C_2$, and so forth. When the dependence on certain key parameters is essential to the analysis, we may highlight this relationship through notation such as $C_2(r_0)$ to indicate that the constant $C_2$ depends specifically on the parameter $r_0$.
	
	We adopt the following notation:
	$$
	\nabla u =\nabla_{\mathbf{x}}u=\left(\partial_{{x}_1}u,\partial_{{x}_2}u,\cdots,\partial_{{x}_d}u\right)=\left(u_{{x}_1},u_{{x}_2},\cdots,u_{{x}_d}\right),\quad \partial_{t}u=u_t=u^\prime.
	$$
	$\Re u$ and $\Im u$ denote the real and imaginary parts of $u$, respectively, and $\mathrm{i}$ denotes $\sqrt{-1}$. The floor function is denoted by $\lfloor x \rfloor$.
	
	For an open set $U \subseteq \mathbb{R}^n$, an integer $k \geq 0$, and $0 < \gamma \leq 1$, the H\" older space $C^{k,\gamma}(\bar{U})$ consists of all functions $u \in C^k(\bar{U})$ for which the norm
	$$
	\|u\|_{C^{k,\gamma}(\bar{U})} = \sum_{|\alpha| \leq k} \sup_{\bar{U}} |D^{\alpha}u| + \sum_{|\alpha| = k} [D^{\alpha}u]_{C^{\gamma}(\bar{U})}
	$$
	is finite, where $[v]_{C^{\gamma}(\bar{U})}$ denotes the $\gamma$-H\" older seminorm:
	$$
	[v]_{C^{\gamma}(\bar{U})} = \sup_{\substack{\mathbf{x},\mathbf{y} \in U \\ \mathbf{x} \neq \mathbf{y}}} \frac{|v(\mathbf{x})-v(\mathbf{y})|}{|\mathbf{x}-\mathbf{y}|^{\gamma}}.
	$$
	
	Let $1 \leq p \leq \infty$ and let $X$ be a Banach space. The Bochner–Sobolev space $W^{1,p}(0,T;X)$ consists of all functions $u \in L^p(0,T;X)$ such that the first weak derivative $u'$ exists and belongs to $L^p(0,T;X)$. The norm on $W^{1,p}(0,T;X)$ is defined as
	\[
	\|u\|_{W^{1,p}(0,T;X)} = 
	\begin{cases}
		\left(\int_0^T \|u(t)\|_X^p  dt + \int_0^T \|u'(t)\|_X^p  dt\right)^{\frac{1}{p}} & \text{if } 1 \leq p < \infty, \\
		\operatorname{ess\,sup}_{0 \leq t \leq T} \left(\|u(t)\|_X + \|u'(t)\|_X\right) & \text{if } p = \infty.
	\end{cases}
	\]
	Similarly, the Bochner Sobolev space $W^{2,p}(0,T;X)$ consists of all functions $u \in L^p(0,T;X)$ such that both the first and second weak derivatives $u'$ and $u''$ exist and belong to $L^p(0,T;X)$. The norm on $W^{2,p}(0,T;X)$ is given by
	\[
	\|u\|_{W^{2,p}(0,T;X)} = 
	\begin{cases}
		\left(\int_0^T \|u(t)\|_X^p  dt + \int_0^T \|u'(t)\|_X^p  dt + \int_0^T \|u''(t)\|_X^p  dt\right)^{\frac{1}{p}} & \text{if } 1 \leq p < \infty, \\
		\operatorname{ess\,sup}_{0 \leq t \leq T} \left(\|u(t)\|_X + \|u'(t)\|_X + \|u''(t)\|_X\right) & \text{if } p = \infty.
	\end{cases}
	\]
	Equipped with these norms, both $W^{1,p}(0,T;X)$ and $W^{2,p}(0,T;X)$ are Banach spaces.
	
	Let $J_m$ denote the Bessel function of the first kind and $j_m$ the spherical Bessel function. Let $j_{m,s}$ denote the $s$-th positive zero of $J_m$ and $j_{m,s}'$ the $s$-th positive zero of $J_m'$. According to \cite{Abramowitz} and \cite{LiuZou}, the following properties hold:
	\begin{equation*}
		m \leq j_{m,1}' < j_{m,1} < j_{m,2}' < j_{m,2} < j_{m,3}' < \cdots.
	\end{equation*}
	The Bessel function admits the Weierstrass factorization:
	\begin{equation*}\label{eq:bessel_factorization}
		J_m(x) = \frac{(|x|/2)^m}{\Gamma(m+1)} \prod_{s=1}^{\infty} \left(1 - \frac{x^2}{j_{m,s}^2}\right).
	\end{equation*}
	For integer or positive $\alpha$, the Bessel function has the power series expansion:
	\begin{equation}\label{eq:bessel_series}
		J_\alpha(x) = \sum_{m=0}^{\infty} \frac{(-1)^m}{m!\,\Gamma(m+\alpha+1)} \left(\frac{x}{2}\right)^{2m+\alpha}, \quad x > 0,
	\end{equation}
	where $\Gamma(x)$ is the gamma function. The relationship between $J_m$ and $j_m$ is given by:
	\begin{equation}\label{eq:bessel_spherical}
		j_m(x) = \sqrt{\frac{\pi}{2x}} J_{m+\frac{1}{2}}(x).
	\end{equation}
	The spherical harmonics $Y_m^l(\theta,\varphi)$ (see \cite[Theorem 2.8]{ColtonKress}) are defined as:
	\begin{equation}\label{eq:spherical_harmonics}
		Y_m^l(\theta, \varphi) = \sqrt{\frac{2m+1}{4\pi} \frac{(m-|l|)!}{(m+|l|)!}} P_m^{|l|}(\cos \theta) e^{i m \varphi},
	\end{equation}
	where $m = 0,1,2,\ldots$, $l = -m,\ldots,m$, and $P_m^{|l|}$ are the associated Legendre functions. The following estimate for Legendre functions will be used in the subsequent analysis.
	\begin{lemma}\label{lem:legendre_bound}
		(see \cite{Lohofer}, Corollary 3). For real $x \in [-1,1]$ and integers $m$, $n$ with $1 \leq m \leq n$,
		\begin{equation}\label{eq:legendre_bound}
			\frac{1}{\sqrt{2.22(m+1)}} < \max_{x \in [-1,1]} |P_n^m(x)| \sqrt{\frac{(n-m)!}{(n+m)!}} < \frac{2^{5/4}}{\pi^{3/4}} \frac{1}{m^{1/4}}.
		\end{equation}
	\end{lemma}
	The Herglotz wave function is defined as
	\begin{equation}\label{Herglotz}
		H_{g}(\mathbf{x})=\int_{\mathbb S^{d-1}} g(\theta) \exp(\mathrm{i} \omega \mathbf{x} \cdot \theta) d\theta, \quad \mathbf{x} \in \mathbb{R}^d,
	\end{equation}
	where $g \in L^2(\mathbb{S}^{d-1})$ is the Herglotz kernel. We now state a key lemma that will be used in the next section.
	
	\begin{lemma}(see~\cite{Weck})\label{Herglotz lemma}
		Let $\mathcal{O}$ be a bounded domain of class $C^{\alpha,1}$, $\alpha \in \mathbb{N} \cup \{0\}$, in $\mathbb{R}^d$. Denote by $\mathbb{H}$ the space of all Herglotz wave functions of the form (\ref{Herglotz}). Define
		\[
		\mathbb{H}(\mathcal{O}) := \left\{\left.u\right|_{\mathcal{O}} : u \in \mathbb{H}\right\}
		\]
		and
		\begin{equation}\label{wave number}
			\mathfrak{H}(\mathcal{O}) := \left\{u \in C^{\infty}(\mathcal{O}) : \Delta u + \omega^2 u = 0 \text{ in } \mathcal{O}\right\}.
		\end{equation}
		Then $\mathbb{H}(\mathcal{O})$ is dense in $\mathfrak{H}(\mathcal{O}) \cap H^{\alpha+1}(\mathcal{O})$ with respect to the $H^{\alpha+1}(\mathcal{O})$-norm.
	\end{lemma}
	
	\begin{remark}\label{remark 2.3}
		The Herglotz wave function $H_g$ defined in \eqref{Herglotz} satisfies the Helmholtz equation $\Delta H_g + \omega^2 H_g = 0$ in $\mathbb{R}^d$. By standard elliptic regularity theory, $H_g$ is analytic with respect to the spatial variable. A key property of Herglotz wave functions is that they are generally not square-integrable over the entire space; specifically, $H_g \in L^{2}_{\text{loc}}(\mathbb{R}^d)$ but $H_g \notin L^{2}(\mathbb{R}^d)$ for non-trivial kernels $g$. 
	\end{remark}

	\section{A general framework for specifying inputs}\label{section 3}
	
	In this section, we establish the general framework for constructing initial and boundary inputs for problems \eqref{1.1} and \eqref{1.2}. To achieve this, we employ spectral analysis, interior transmission problems, and Herglotz approximation techniques to construct a function $u_0$ that will serve as the foundation for our desired inputs.
	
	\begin{theorem}\label{thm:main_result}
		For any predetermined positive parameters $\varepsilon < 1$, $r_0$ sufficiently small, and $n$ distinct points $\mathbf{x}_1,\dots,\mathbf{x}_n \in \partial D$, we can properly choose corresponding $n$ balls satisfying:
		\begin{enumerate}
			\renewcommand{\labelenumi}{(\roman{enumi})}
			\item For problem \eqref{1.1} $($bounded domain $\Omega$ with $D \Subset \Omega):$ $B_{r_0}(\mathbf{y}_i) \Subset \Omega \setminus \overline{D}$,
			\item For problem \eqref{1.2} $($full space $\mathbb{R}^d ):$ $B_{r_0}(\mathbf{y}_i) \Subset \mathbb{R}^d \setminus \overline{D}$,
		\end{enumerate}
		such that $|\mathbf{x}_i - \mathbf{y}_i| = 2r_0$, the balls are mutually disjoint and
		$$
		\operatorname{dist}(\partial D, B_{r_0}(\mathbf{y}_i))>\frac{r_0}{2},\ \forall i=1,2...,n.
		$$
		For any fixed positive constant $\omega$ described in Item 10 in \eqref{1.1} or \eqref{1.2},  there exists an integer $M := M(r_0,\omega)$ and a smooth function $u_0$ with the following properties:
		\begin{enumerate}
			\item $u_0$ has the form$:$
			\begin{equation*}
				u_0(\mathbf{x},t) = H_g(\mathbf{x})\exp(\mathrm{i}\omega t),
			\end{equation*}
			where $ H_g(\mathbf{x})$ is a Herglotz wave function. 
			
			\item $u_0$ satisfies the wave equation$:$
			\[
			\partial_{tt}^2 u_0 - \Delta u_0 = 0 \quad \text{in } \mathbb{R}^d \times [0,\infty).
			\]
			
			\item It exhibits small amplitude in $\overline{D}:$
			\begin{equation}\label{u0 smalles}
				|u_0(\mathbf{x},t)| < C_1\varepsilon \quad \text{for } (\mathbf{x},t) \in \overline{D} \times [0,\infty),
			\end{equation}
			where $C_1$ depends on $D$ and $d$.
			
			\item For any predetermined integer $m \geqslant M(r_0,\omega)$, there exist points $\mathbf{y}^*_i \in \partial B_{r_0}(\mathbf{y}_i)$  ($i=1,\dots,n$) with large amplitude $u_0(\mathbf{y}^*_i,t)$ such that
			\begin{align}
				|u_0(\mathbf{y}^*_i,t)| &> \frac{\sqrt{2m+2}\,r_0^{-1}}{2\sqrt{2\pi}} - C_2\varepsilon, & d &= 2, \label{2 d} \\
				|u_0(\mathbf{y}^*_i,t)| &> \frac{\sqrt{2m+3}\,r_0^{-3/2}}{16} - C_2\varepsilon, & d &= 3, \label{3 d}
			\end{align}
			for all $t \in [0,\infty)$, where $C_2$ depends on $r_0$, $D$, and $d$.
		\end{enumerate}
	\end{theorem}
	
	\begin{remark}\label{cf 3.5}
		From now and remaining of this paper, we denote by $\varepsilon$, $r_0$ the parameters, and by $u_0$ the function, specified in Theorem \ref{thm:main_result}.
	\end{remark}
	
	\begin{remark}\label{remark 3.2}
		For notation convenience, we denote $C_2 := C_2(r_0)$. In particular, by \eqref{important estimate for H_g}, \eqref{3.17}, and the Sobolev embedding theorem, $C_2 \to \infty$ as $r_0 \to 0$. The constants $C_1$ and $C_2$ will be used to determine $\varepsilon$, $r_0$, and $m$ in the proofs of Theorems \ref{thm:main} and \ref{thm:main2}.
	\end{remark}
	
	\begin{remark}
		The construction distinguishes between two physical settings:
		\begin{enumerate}
			\renewcommand{\labelenumi}{(\roman{enumi})}
			\item For the linear problem \eqref{1.1} in the bounded domain $\Omega$ (with $D \Subset \Omega$), the balls satisfy $B_{r_0}(\mathbf{y}_i) \Subset \Omega \setminus \overline{D}$.
			\item For the nonlinear problem \eqref{1.2} in the full space $\mathbb{R}^d$, the balls satisfy $B_{r_0}(\mathbf{y}_i) \Subset \mathbb{R}^d \setminus \overline{D}$.
		\end{enumerate}
		Importantly, the selection of these balls is highly flexible they can be chosen arbitrarily within the respective domains, subject only to the above containment conditions. Despite these domain differences, the constructed $u_0$ shares identical properties in both cases. For simplicity, we do not distinguish between the corresponding $u_0$ functions in subsequent analysis.
	\end{remark}
	
	Prior to proving Theorem \ref{thm:main_result}, we establish the following auxiliary result:
	\begin{lemma}\label{I_1,I_2,I_3,I_4}
		For any fixed $r_0$, the series $I_{i}(m,\omega,r_0)\rightarrow 0$ as $m \rightarrow \infty$ for $i=1,2,3,4$, where
		\begin{align*}
			I_1(m,\omega ,r_0) &:= \sum_{k=1}^{\infty}\frac{(-1)^k\Gamma(m+\frac{3}{2})\sqrt{2m+3} }{k!\Gamma(m+k+\frac{3}{2})}\left(\frac{\omega r_0}{2}\right)^{2k}r_0^{-\frac{3}{2}}, \\
			I_2(m,\omega,r_0) &:= \sum_{k=1}^{\infty}\sum_{k_1+k_2=k}\left(\frac{\Gamma^2(m+\frac{3}{2})}{k_1!k_2!\Gamma(m+k_1+\frac{3}{2})\Gamma(m+k_2+\frac{3}{2})}\right)  \\
			& \qquad \times(-1)^k\left(\frac{\omega r_0}{2}\right)^{2k}\frac{2m+3}{2m+2k+3}, \\
			I_3(m,\omega,r_0) &:= \sum_{k=1}^{\infty}\frac{(-1)^k}{k!}\frac{\Gamma(m+1)\sqrt{2m+2}}{\Gamma(m+k+1)}\left(\frac{\omega r_0}{2}\right)^{2k}r^{-1}_0, \\
			I_4(m,\omega,r_0) &:= \sum_{k=1}^{\infty}\left(\sum_{k_1+k_2=k}\frac{\Gamma^{2}(m+1)}{k_1!k_2!\Gamma(m+k_1+1)\Gamma(m+k_2+1)}\right) \nonumber \\
			& \qquad \times(-1)^k\left(\frac{\omega r_0}{2}\right)^{2k}\frac{2m+2}{2m+2k+2}.
		\end{align*}
		\begin{proof}
			Set $x=\omega r_0$ and $\nu=m+\frac{1}{2}$, so that $\Gamma(m+k+\frac{3}{2})=\Gamma(\nu+k+1)$ and $\sqrt{2m+3}=\sqrt{2\nu+2}$. Using \eqref{eq:bessel_series}, we obtain
			\begin{equation}\label{I_1 1}
				I_1=\sqrt{2\nu +2} r_0^{-\frac{3}{2}}[\Gamma(\nu+1)(\frac{x}{2})^{-\nu}J_{\nu}(x)-1].
			\end{equation}
			Let $\beta=(\frac{x}{2})^2$ and $\gamma_k(\nu)=\frac{1}{(\nu+k)(\nu+k-1)\cdots(\nu+1)}$. By the Gamma function recurrence and \eqref{eq:bessel_series},
			\begin{equation}\label{I_1 2}
				\begin{aligned}
					|\Gamma(\nu+1)(\frac{x}{2})^{-\nu}J_{\nu}(x)-1|&=|\sum_{k=1}^{\infty}\frac{(-1)^k}{k!}\beta^k\gamma_k(\nu)|\\
					&\leqslant\exp(\frac{\beta}{\nu})-1\leqslant \frac{\beta}{\nu}\exp(\frac{\beta}{\nu})\leqslant \frac{\beta}{\nu}\exp(\beta),
				\end{aligned}
			\end{equation}
			where we used $|\gamma_k(\nu)|\leqslant \nu^{-k}$. Combining \eqref{I_1 1} and \eqref{I_1 2} yields
			\begin{equation*}
				|I_1|\leqslant \frac{\sqrt{2\nu +2}}{\nu} r_0^{-\frac{3}{2}}\beta\exp(\beta) =\frac{\sqrt{2m +3}}{m+\frac{1}{2}} r_0^{-\frac{3}{2}}(\frac{\omega r_0}{2})^2\exp(\frac{\omega r_0}{2})^2\rightarrow 0 
			\end{equation*}
			as $m \rightarrow \infty$.
			
			For $I_2$, rewrite it as
			\[
			\begin{aligned}
				I_2\left(m, \omega, r_0\right)=(2 m+3) \sum_{k=1}^{\infty} \frac{(-1)^k\left(\omega r_0 / 2\right)^{2 k}}{2 m+2 k+3} \cdot S_k(m),
			\end{aligned}
			\]
			where 
			\[
			S_k(m)=\sum_{\substack{k_1+k_2=k \\ k_1, k_2 \geq 0}} \frac{1}{k_{1}!k_{2}!} \cdot \frac{\Gamma^2\left(m+\frac{3}{2}\right)}{\Gamma\left(m+k_1+\frac{3}{2}\right) \Gamma\left(m+k_2+\frac{3}{2}\right)}.
			\]
			For fixed $j \geq 0$, the Gamma recurrence implies
			\[
			\frac{\Gamma\left(m+\frac{3}{2}\right)}{\Gamma\left(m+j+\frac{3}{2}\right)} \leq \frac{1}{\left(m+\frac{3}{2}\right)^j}.
			\]
			Applying this to $S_k(m)$ and noting $\sum_{k_1+k_2=k} \frac{1}{k_{1}!k_{2}!}=\frac{2^k}{k!}$, we get
			\[
			S_k(m) \leq \frac{2^k}{k!\left(m+\frac{3}{2}\right)^k}.
			\]
			Thus,
			\begin{equation*}
				\begin{aligned}
					\left|I_2(m)\right| &\leq(2 m+3) \sum_{k=1}^{\infty} \frac{1}{2 m+2 k+3} \cdot \frac{\left(\omega^2 r_0^2 / 2\right)^k}{k!\left(m+\frac{3}{2}\right)^k}\\
					&\leqslant \sum_{k=1}^{\infty} \frac{\left(\omega^2 r_0^2 / 2\right)^k}{k!\left(m+\frac{3}{2}\right)^k}=\exp(\frac{\omega^2 r_0^2}{2m+3})-1\rightarrow 0
				\end{aligned}
			\end{equation*}
			as $m \rightarrow \infty$. The same approach applies to $I_3$ and $I_4$, so their proofs are omitted.
			
			The proof is complete.
		\end{proof}
	\end{lemma}
	
	\begin{proof}[Proof of Theorem \ref{thm:main_result}]
		The proof proceeds in two steps.
		
		\medskip
		\textbf{Step 1: Construction of the transmission eigenfunction and its properties.}
		
		Consider the interior transmission problem (see \cite{ColtonKress}, p.~319):
		\begin{equation}\label{Interior Transmission Problem}
			\left\{\begin{array}{@{}ll@{}}
				(\Delta + \omega^2 \mathfrak{n}^2) w = 0 & \text{in } \displaystyle\bigcup_{i=1}^{n} B_{r_0}(\mathbf{y}_i), \\
				(\Delta + \omega^2) v = 0 & \text{in } \displaystyle\bigcup_{i=1}^{n} B_{r_0}(\mathbf{y}_i), \\
				w = v,\ \dfrac{\partial w}{\partial \nu} = \dfrac{\partial v}{\partial \nu} & \text{on } \partial\left(\displaystyle\bigcup_{i=1}^{n} B_{r_0}(\mathbf{y}_i)\right),
			\end{array}\right.
		\end{equation}
		where $\nu$ denotes the outward unit normal. A pair $(w,v)$ satisfying \eqref{Interior Transmission Problem} is called a transmission eigenfunction corresponding to the transmission eigenvalue $\omega$. We analyze explicit solutions in both two-dimensional  and three-dimensional case, showing that $v(\mathbf{x})$ attains its maximum simultaneously on each $\partial B_{r_0}(\mathbf{y}_i)$ ($i=1,\dots,n$) with arbitrarily large magnitude for large predetermined parameter $m$.
		
		\item[\textbf{Case 1}: $d=3.$] For a fixed positive integer $m$ to be determined, in polar coordinates with $\mathbf{x} \in B_{r_0}(\mathbf{y}_i)$ ($i=1,\dots,n$):
		\begin{equation*}
			w(\mathbf{x}) = \alpha_{m}^{l} j_m(\omega \mathfrak{n} r_i) Y_{m}^{l}(\theta_i, \varphi_i), \quad
			v(\mathbf{x}) = \beta_{m}^{l} j_m(\omega r_i) Y_{m}^{l}(\theta_i, \varphi_i), \quad -m \leqslant l \leqslant m,
		\end{equation*}
		where $r_i = |\mathbf{x}-\mathbf{y}_i|$, $\theta_i = \arccos((x_3-y_{i3})/r_i)$, $\varphi_i = \arctan((x_2-y_{i2})/(x_1-y_{i1}))$. To satisfy the boundary conditions on $\partial B_{r_0}(\mathbf{y}_i)$, we set $\alpha^l_m = [j_m(\omega r_0)/j_m(\omega \mathfrak{n} r_0)] \beta^l_m$ with $\mathfrak{n}$ solving
		\begin{equation*}
			\mathfrak{n} j_{m}^\prime(\omega \mathfrak{n} r_0) j_{m}(\omega r_0) - j_{m}^\prime(\omega r_0) j_{m}(\omega \mathfrak{n} r_0) = 0.
		\end{equation*}
		Normalizing $v$ such that $\|v\|_{L^2(B_{r_0}(\mathbf{y}_i))} = 1$ for all $i$, we obtain
		\begin{equation*}
			\beta_m^l = \sqrt{\frac{2}{\pi}} \frac{\omega^{\frac{3}{2}}}{\sqrt{\int_0^{\omega r_0} J^2_{m+\frac{1}{2}}(r) r  dr}}.
		\end{equation*}
		Combining these results yields
		\begin{equation}\label{v in d=3}
			v(\mathbf{x}) = \frac{\omega r_i^{-\frac{1}{2}} J_{m+\frac{1}{2}}(\omega r_i)}{\sqrt{\int_0^{\omega r_0} J_{m+\frac{1}{2}}^2(r) r  dr}} Y_m^l(\theta_i, \varphi_i), \quad \mathbf{x} \in B_{r_0}(\mathbf{y}_i).
		\end{equation}
		By using recurrence relation for Bessel function, we have
		\begin{equation*}
			\left(\frac{J_{m+\frac{1}{2}}(\omega r)}{r^{\frac{1}{2}}}\right)^{\prime}=\frac{m J_{m+\frac{1}{2}}(\omega r)-\omega r J_{m+\frac{3}{2}}(\omega r) }{r^{\frac{3}{2}}}.
		\end{equation*}
		Let $m>\omega r_0$ and use the fact that $J_{m+\frac{1}{2}}(r)>J_{m+\frac{3}{2}}(r)~ ,\forall r \in (0,m]$, then $J_{m+\frac{1}{2}}(\omega r) r^{-\frac{1}{2}}$ is increasing on $\left(0, r_0\right]$. Consequently, by \eqref{eq:bessel_series}, we calculate
		\begin{equation*}
			\begin{aligned}
				&\max _{r_i \in\left(0, r_0\right]}\frac{\omega {r_i^{-\frac{1}{2}}} J_{m+\frac{1}{2}}(\omega r_i)}{\sqrt{\int_0^{\omega r_0} J_{m+\frac{1}{2}}^2(r) r d r}} =\frac{\omega{r_0^{-\frac{1}{2}}} J_{m+\frac{1}{2}}(\omega r_0)}{\sqrt{\int_0^{\omega r_0} J_{m+\frac{1}{2}}^2(r) r d r}}\\
				&=\frac{\sum_{k=0}^{\infty}\frac{(-1)^{k}\omega^{m+2k+\frac{3}{2}}r_{0}^{m+2k}}{k!2^{m+2k+\frac{1}{2}}\Gamma(m+k+\frac{3}{2})}}{\sqrt{\sum_{k=0}^{\infty}\sum_{k_1+k_2=k}\frac{1}{k_1{!}k_2{!}\Gamma(m+k_1+\frac{3}{2})\Gamma(m+k_2+\frac{3}{2})}\frac{(-1)^k(\omega r_0)^{2m+2k+3}}{2^{2m+2k+1}2m+2k+3}}}\\
				&=\frac{\sqrt{2m+3}r_0^{-\frac{3}{2}}+\sum_{k=1}^{\infty}\frac{(-1)^k\Gamma(m+\frac{3}{2})\sqrt{2m+3} }{k!\Gamma(m+k+\frac{3}{2})}\left(\frac{\omega r_0}{2}\right)^{2k}r_0^{-\frac{3}{2}}}{\sqrt{1+\sum_{k=1}^{\infty}\sum_{k_1+k_2=k}\left(\frac{\Gamma^2(m+\frac{3}{2})}{k_1!k_2!\Gamma(m+k_1+\frac{3}{2})\Gamma(m+k_2+\frac{3}{2})}\right)\times\frac{(-1)^k(\omega r_0/2)^{2k}(2m+3)}{(2m+2k+3)}}}\\
				&=\frac{\sqrt{2m+3}r_0^{-\frac{3}{2}}+I_1(m,\omega, r_0)}{\sqrt{1+I_2(m,\omega, r_0)}}.
			\end{aligned}
		\end{equation*}
		By Lemma \ref{I_1,I_2,I_3,I_4}, $I_1, I_2 \to 0$ as $m \to \infty$ for fixed $r_0$. Thus, there exists $M_1(\omega,r_0)$ such that for predetermined parameter $m \geqslant M_1$,
		\begin{equation}\label{upper bound in d=3}
			\frac{J_{m+\frac{1}{2}}(\omega r_0) \omega r_0^{-\frac{1}{2}}}{\sqrt{\int_0^{\omega r_0} J_{m+\frac{1}{2}}^2(r) r  dr}} > \frac{1}{2} \sqrt{2m+3}  r_0^{-\frac{3}{2}}.
		\end{equation}
		Combining \eqref{eq:spherical_harmonics}, \eqref{v in d=3}, \eqref{upper bound in d=3}, and spherical harmonic properties from Lemma \ref{lem:legendre_bound}, there exist points $\mathbf{y}^*_i \in \partial B_{r_0}(\mathbf{y}_i)$ satisfying
		\begin{equation}\label{3.12}
			|v(\mathbf{y}^*_i)| > \frac{1}{16} \sqrt{2m+3}  r_0^{-\frac{3}{2}}.
		\end{equation}
		
		\item[\textbf{Case 2}: $d=2.$] The two-dimensional case follows analogously. For $\mathbf{x} \in B_{r_0}(\mathbf{y}_i)$ ($i=1,\dots,n$):
		\begin{equation*}
			w(\mathbf{x}) = \alpha_m J_m(\omega \mathfrak{n} r_i) e^{\mathrm{i}m \theta_i}, \quad
			v(\mathbf{x}) = \beta_m J_m(\omega r_i) e^{\mathrm{i}m \theta_i},
		\end{equation*}
		with $r_i = |\mathbf{x}-\mathbf{y}_i|$, $\theta_i = \arctan((x_2-y_{i2})/(x_1-y_{i1}))$. Setting $\alpha_m = [J_m(\omega r_0)/J_m(\omega \mathfrak{n} r_0)] \beta_m$ and $	\beta _m=\frac{1}{\sqrt{2\pi}}\frac{\omega}{\sqrt{ \int_{0}^{\omega r_0}J^2_m(r)rdr}}$  yields 
		\begin{equation}
			v(\mathbf{x}) = \frac{1}{\sqrt{2\pi}} \frac{\omega J_m(\omega r_i)}{\sqrt{\int_{0}^{\omega r_0} J_{m}^2(r) r  dr}} e^{\mathrm{i}m \theta_i}.
		\end{equation}
		For $m > \omega r_0$, by \eqref{eq:bessel_series}, we obtain
		\begin{equation*}
			\max_{r_i \in [0,r_0]} \frac{\omega J_m(\omega r_i)}{\sqrt{\int_{0}^{\omega r_0} J_{m}^2(r) r  dr}} = \frac{\sqrt{2m+2} r^{-1}_0 + I_3(m,\omega,r_0)}{\sqrt{1 + I_4(m,\omega,r_0)}}.
		\end{equation*}
		By Lemma \ref{I_1,I_2,I_3,I_4}, $I_3, I_4 \to 0$ as $m \to \infty$ for fixed $r_0$. Then, for any predetermined $m > M_2(\omega, r_0)$, there exist points $\mathbf{y}^{*}_{i} \in \partial B{r_0}(\mathbf{y}_i)$ such that
		\begin{equation}\label{3.14}
			|v(\mathbf{y}^*_i)| \geqslant \frac{\sqrt{2m+2} r^{-1}_0}{2\sqrt{2\pi}}.
		\end{equation}
		Defining $M = \max\{M_1(\omega,r_0), M_2(\omega,r_0), \lfloor \omega r_0 \rfloor + 1\}$, both \eqref{3.12} and \eqref{3.14} hold for predetermined parameter $m \geqslant M$.

		\medskip
		\textbf{Step 2: Herglotz approximation and construction of $u_0$.}
		
		Since bounded Lipschitz domains admit smooth outer approximations \cite{Antonini}, let $D_{r_0/2}$ be a $C^\infty$-smooth domain satisfying
		\begin{equation*}
			D \Subset D_{r_0/2} \quad \text{and} \quad \operatorname{dist}(D, \partial D_{r_0/2}) < r_0/2.
		\end{equation*}
		As established in Theorem \ref{thm:main_result}, the sets \(D_{r_0/2}\) and \(B_{r_0}(\mathbf{y}_i)\) are disjoint for all \(i\). Then we define $\mathcal{O} = D_{r_0/2} \cup \bigcup_{i=1}^{n} B_{r_0}(\mathbf{y}_i)$. Consider the interior transmission problem
		\begin{equation}\label{Interior transmission problem 2}
			\left\{\begin{array}{@{}ll@{}}
				(\Delta + \omega^2 \mathfrak{n}^2) w_1 = 0 & \text{in } \mathcal{O}, \\
				(\Delta + \omega^2) v_1 = 0 & \text{in } \mathcal{O}, \\
				w_1 = v_1,\ \dfrac{\partial w_1}{\partial \nu} = \dfrac{\partial v_1}{\partial \nu} & \text{on } \partial \mathcal{O}.
			\end{array}\right.
		\end{equation}
		The pair $(w_1, v_1) = \chi_{\bigcup_{i=1}^n B_{r_0}(\mathbf{y}_i)} (w,v) + \chi_{D_{r_0/2}} (0,0)$ satisfies \eqref{Interior transmission problem 2} by construction. Applying Lemma \ref{Herglotz lemma} to $\chi_{\bigcup_{i=1}^n B_{r_0}(\mathbf{y}_i)} v + \chi_{D_{r_0/2}} \cdot 0$, there exists $g \in L^2(\mathbb{S}^{d-1})$ such that the Herglotz wave function
		\begin{equation*}
			H_g(\mathbf{x}) = \int_{\mathbb{S}^{d-1}} g(\theta) \exp(\mathrm{i} \omega \mathbf{x} \cdot \theta)  d\theta
		\end{equation*}
		satisfies
		\begin{equation}\label{eq3.14}
			\left\| H_g - \left( \chi_{\bigcup_{i=1}^n B_{r_0}(\mathbf{y}_i)} v + \chi_{D_{r_0/2}} \cdot 0 \right) \right\|_{H^5(\mathcal{O})} < \varepsilon.
		\end{equation}
		This implies
		\begin{equation}\label{important estimate for H_g}
			\| H_g - v \|_{H^5(B_{r_0}(\mathbf{y}_i))} < \varepsilon \quad (i=1,\dots,n) \quad \text{and} \quad \| H_g \|_{H^5(D)} < \varepsilon.
		\end{equation}
		Define
		\begin{equation}\label{u_0}
			u_0(\mathbf{x},t) = H_g(\mathbf{x}) \exp(\mathrm{i} \omega t).
		\end{equation}
		Properties 1 and 2 follow directly from \eqref{u_0}. By the Sobolev embedding theorem and \eqref{important estimate for H_g}, there exist constants $C_1$ (depending on $D, d$) and $C_2$ (depending on $r_0, D, d$) such that
		\begin{equation}\label{3.18}
			\| H_g \|_{L^{\infty}(D)} \leqslant C_1 \varepsilon,
		\end{equation}
		and
		\begin{equation}\label{3.17}
			\| H_g - v \|_{L^{\infty}(B_{r_0}(\mathbf{y}_i))} \leqslant C_2 \varepsilon \quad (i=1,\dots,n).
		\end{equation}	
		Combining \eqref{3.12}, \eqref{3.14}, \eqref{u_0}, \eqref{3.17}, and \eqref{3.18} with the continuity of $H_g$ and $v$, Properties 3 and 4 are verified.
		
		The proof is complete.
	\end{proof}
	
	\begin{remark}
		The approximation of $\chi_{\bigcup_{i=1}^n B_{r_0}(\mathbf{y}_i)} v + \chi_{D_{r_0/2}} \cdot 0$ in the $H^5$ norm required in \eqref{eq3.14} is valid due to the $C^\infty$ regularity of $\partial D_{r_0/2}$. This construction serves a critical purpose: the resulting estimate will be employed in the subsequent section to separately control the source term $F_1$ in equation~\eqref{remaining term},  the source term $F_2$ in equation~\eqref{nonlinear hyperbolic equation with source term and zero initial condition}, and all spatial derivatives $\partial_{\mathbf{x}}^{\alpha} u_0$ for $|\alpha|\leqslant 3$. Detailed justification for this approach is provided in \eqref{estimate on F_1} and Lemma~\ref{Lemma 5.1}.
	\end{remark}
	
	The initial and boundary data for problem \eqref{1.1} are formally specified as:
	\begin{equation}\label{initial and boundary condition}
		\begin{aligned}
			& \phi_1(\mathbf{x}) = \left. u_0(\mathbf{x},t) \right|_{\Omega \times \{ t=0 \}} = H_{g}(\mathbf{x})|_{\Omega}, \\
			& \phi_2(\mathbf{x}) = \partial_t u_0 \left. (\mathbf{x},t) \right|_{\Omega \times \{ t=0 \}} = \mathrm{i}\omega H_{g}(\mathbf{x})|_{\Omega}, \\
			& \psi(\mathbf{x},t) = u_0 \left. (\mathbf{x},t) \right|_{\partial \Omega \times [0,T]} = \exp(\mathrm{i}\omega t)H_g(\mathbf{x})|_{\partial \Omega \times [0,T]}.
		\end{aligned}
	\end{equation}
	Similarly, the initial data for problem \eqref{1.2} are defined as:
	\begin{equation}\label{initial condition}
		\begin{aligned}
			& \Phi_1(\mathbf{x}) = \left. u_0(\mathbf{x},t) \right|_{\mathbb{R}^d \times \{ t=0 \}} = H_{g}(\mathbf{x})|_{\mathbb{R}^d}, \\
			& \Phi_2(\mathbf{x}) = \partial_t u_0 \left. (\mathbf{x},t) \right|_{\mathbb{R}^d \times \{ t=0 \}} = \mathrm{i}\omega H_{g}(\mathbf{x})|_{\mathbb{R}^d}.
		\end{aligned}
	\end{equation}
	
\begin{remark}
	It is essential to recognize that the constructions in \eqref{initial and boundary condition} and \eqref{initial condition} remain abstract at this stage. The function $u_0$—and consequently all associated data—depends critically on the yet-to-be-determined parameters $\varepsilon$, $m$, $r_0$, and the configuration of points $\mathbf{x}_1, \dots, \mathbf{x}_n$ on $\partial D$. A concrete realization of $u_0$ requires explicitly specifying these parameters, which are intrinsically determined by the prescribed large constant $\mathcal{M}$ and the spatial locations of the points $\mathbf{x}_1, \dots, \mathbf{x}_n$. In the next section, we will systematically fix $\varepsilon$, $m$, and $r_0$ in terms of $\mathcal{M}$ and the given boundary point configuration, ultimately yielding a well-defined $u_0$ and corresponding initial and boundary inputs suitable for rigorous analysis.
\end{remark}
	
	\section{Proof of Theorem \ref{1.1}}\label{sec:4}
	
	In this section, we present the complete proof of Theorem~\ref{thm:main}. Our strategy involves analyzing an auxiliary problem \eqref{remaining term} with zero initial-boundary conditions and a source term, derived from the linear second-order hyperbolic equation \eqref{1.1}. We establish well-posedness for this  auxiliary problem via the standard Galerkin method and demonstrate that its solution exhibits small H\" older norm on $\Omega \setminus \overline{D}$. Finally, by integrating Theorems \ref{thm:main_result} and \ref{3.1} and a delicate balancing of the parameters $\varepsilon$, $r_0$, and $m$, we precisely generate the quasi-singularity phenomenon central to our main result.
	
	Let $\mathcal{U} : \Omega \times [0,T] \rightarrow \mathbb{C}$ satisfy 
	\begin{equation}\label{remaining term}
		\begin{cases}
			L_1\mathcal{U} = F_1 & \text{in } \Omega \times (0,T], \\
			\mathcal{U}(\mathbf{x},0) = 0, \quad \partial_t \mathcal{U}(\mathbf{x},0) = 0 & \text{in } \Omega, \\
			\mathcal{U}(\mathbf{x},t) = 0 & \text{on } \partial\Omega \times [0,T],
		\end{cases}
	\end{equation}
	where source term
	\begin{equation}\label{F_1}
		F_1(\mathbf{x}, t)=-\left[\nabla \cdot\left((\mathbf{I}-\mathbf{A}_1) \nabla u_0\right)-\mathbf{b}_1 \cdot \nabla u_0-c_1 u_0\right].
	\end{equation}
	Recall that $\mathbf{A}_1 - \mathbf{I}, \mathbf{b}_1(\cdot,t), c_1(\cdot,t)$ have compact support in $D$ for $t\in[0,T]$, we have 
	\begin{equation*}
		\operatorname{supp}F_1(\cdot,t)\subset D,\ \forall t\in [0,T].
	\end{equation*}
	Furthermore, combining \eqref{important estimate for H_g}, \eqref{u_0}, \eqref{F_1} and the regularity condition on $\mathbf{A}_1,\mathbf{b}_1,c_1$, we have the following estimate for $F_1$,
	\begin{equation}\label{estimate on F_1}
		\| F_1 \|_{H^2(0,T;L^2(\Omega))} + \| F_1 \|_{H^1(0,T;H^1(\Omega))} \leqslant C\|H_g\|_{H^3(D)} \leqslant C\varepsilon
	\end{equation}
	where the spatial derivatives of \( u_0 \) in \( F_1 \) are up to second order, but the space-time Sobolev norm estimate requires third-order spatial derivatives of \( H_g \) due to the spatial gradient component in \( H^1(\Omega) \). Hence, the \( H^3 \) norm of \( H_g \) suffices to bound all terms in the estimate. The constant $C$ depends on $\omega,d,T,D$, $\|\mathbf{A}_1\|_{W^{2,\infty}(\Omega)},\|\mathbf{b}_{1},c_1\|_{W^{1,\infty}(0,T; W^{1,\infty}(\Omega))}$ and $ \|\mathbf{b}_{1},c_1\|_{W^{2,\infty}(0,T;L^{\infty}(\Omega))}$.
	
	Next, we employ the Galerkin method to establish the well-posedness of problem \eqref{remaining term} for {complex-valued solutions}. Unlike Evans' treatment of {real-valued solutions} under {$C^1$-smooth coefficient assumptions} for $\mathbf{A}_1$, $\mathbf{b}_1$, and $c_1$~\cite{Evans}, our weaker regularity requirements necessitate a complete rederivation. For completeness, we provide the full detailed analysis below.
	\begin{theorem}\label{3.1}
		For any $T$ with $0<T<\infty$, equation \eqref{remaining term} admits a unique solution $\mathcal{U}$ satisfying
		$$ 	\mathcal{U}\in L^{\infty}\left(0,T;H^{1}_{0}(\Omega)\right),\quad  \mathcal{U}^\prime \in L^{\infty}\left(0,T;L^{2}(\Omega)\right) ,$$ 
		with the following properties:
		\begin{enumerate}
			\item Improved H\"older regularity on $\Omega\backslash \overline{D}:$ for any  Lipschtiz domain $\Omega^\prime \Subset \Omega\backslash \overline{D}$,
			\begin{equation*}
				\mathcal{U} \in L^{\infty}(0,T;C^{1,\frac{1}{2}}(\overline{\Omega}^{\prime})). 
			\end{equation*} 
			\item  Smallness estimate:
			\begin{equation*}
				\|\mathcal{U}\|_{L^{\infty}(0,T;C^{1,\frac{1}{2}}(\overline{\Omega}^{\prime})) }\leqslant C_3 \varepsilon ,
			\end{equation*}
			where the constant $C_3$ depends on where the constant $C_3$ depends on $d$, $T$, $\theta$, $D$, $\Omega^\prime$, $\Omega$, $\|\mathbf{A}_1\|_{W^{2,\infty}(\Omega)}$, $\|\mathbf{b}_1\|_{W^{1,\infty}(0,T; W^{1,\infty}(\Omega))}$, $\|c_1\|_{W^{1,\infty}(0,T; W^{1,\infty}(\Omega))}$, $\|\mathbf{b}_1\|_{W^{2,\infty}(0,T; L^{\infty}(\Omega))}$, and $\|c_1\|_{W^{2,\infty}(0,T; L^{\infty}(\Omega))}$.
			
		\end{enumerate}
	\end{theorem}
	
	\begin{proof}
		We introduce the following notations. 	Given complex-valued $\mathcal{U}, \mathcal{V} \in H^1(\Omega)$, write
		\begin{equation*}
			(\mathcal{U}, \mathcal{V})_{L^2}:=\int_\Omega \mathcal{U} \bar{\mathcal{V}} d \mathbf{x}
		\end{equation*}
		and set 	
		\begin{equation*}
			\begin{aligned}
				&A[\mathcal{U}, \mathcal{V};t]:=\int_\Omega \sum_{i, j=1}^d a^{i j}_1(\cdot) \mathcal{U}_{\mathrm{x}_i} \overline{\mathcal{V}}_{\mathrm{x}_j},\\
				&B[\mathcal{U}, \mathcal{V};t]:=\int_\Omega \sum_{i, j=1}^d a^{i j}_1(\cdot) \mathcal{U}_{\mathrm{x}_i} \overline{\mathcal{V}}_{\mathrm{x}_j}+\sum_{i=1}^d b^i_1(\cdot,t) \mathcal{U}_{\mathrm{x}_i} \overline{\mathcal{V}}+c_1(\cdot,t) \mathcal{U} \overline{\mathcal{V}} d \mathrm{x}.\\
			\end{aligned}
		\end{equation*}
		The proof is divided into three steps.

		\medskip
		\textbf{Step 1: Galerkin approximation.}  
		Choose smooth complex-valued functions $\{w_k\}_{k=1}^{\infty}$ satisfying:
		\begin{equation}\label{basis}
			\begin{aligned}
				&\text{(i) Orthonormal basis in } L^2(\Omega;\mathbb{C}) ,\\
				&\text{(ii) Orthogonal basis in } H^1_0(\Omega;\mathbb{C}).
			\end{aligned}
		\end{equation}
		For fixed $m \in \mathbb{N}$, define $\mathcal{U}_m(t) = \sum_{k=1}^{m} d_m^k(t) w_k$ with initial conditions
		\begin{equation}\label{zero condition}
			d_m^k(0) = 0, \quad d_m^{k\prime}(0) = 0 \quad (k=1,\dots,m),
		\end{equation}
		satisfying the variational equation:
		\begin{equation}\label{eq 3.6}
			\left( \mathcal{U}^{\prime\prime}_m(t), w_k \right)_{L^2} + B[\mathcal{U}_m, w_k; t] = \left( F_1(\cdot,t), w_k \right)_{L^2}, \quad k=1,\dots,m.
		\end{equation}
		By \eqref{basis}, this reduces to a complex ODE system:
		\begin{equation*}
			\mathbf{d}''_m(t) + \mathbf{E}(t)\mathbf{d}_m(t) = \mathbf{f}(t), \quad 
			\mathbf{d}_m(0) = \mathbf{d}'_m(0) = \mathbf{0},
		\end{equation*}
		where $\mathbf{f} = \left( (F_1, w_1)_{L^2}, \dots, (F_1, w_m)_{L^2} \right)^\top$, $\mathbf{d}_m = (d_m^1, \dots, d_m^m)^\top$, $\mathbf{E} = (e^{kl})_{m\times m}$ with $e^{kl} := B[w_l, w_k; t]$.
		
		This system is equivalent to a first-order real ODE:
		\begin{equation*}
			\mathbf{Y}'(t) = \mathbf{J}(t)\mathbf{Y}(t) + \mathbf{F}(t),\quad \mathbf{Y}(0)=\mathbf{0},
		\end{equation*}
		where $\mathbf{Y} = \left( \operatorname{\Re}\mathbf{d}_m, \operatorname{\Im}\mathbf{d}_m, \operatorname{\Re}\mathbf{d}_m', \operatorname{\Im}\mathbf{d}_m' \right)^\top$, $\mathbf{F} = \left( \mathbf{0}_{2m}, \operatorname{\Re}\mathbf{f}, \operatorname{\Im}\mathbf{f} \right)^\top$, and
		$
		\mathbf{J} = \begin{pmatrix}
			\mathbf{0}_{2m} & \mathbf{I}_{2m} \\
			\mathbf{M}_{2m} & \mathbf{0}_{2m}
		\end{pmatrix}
		$ with $\mathbf{M}_{2m} = \begin{pmatrix}
			\operatorname{\Re}\mathbf{E} & -\operatorname{\Im}\mathbf{E} \\
			\operatorname{\Im}\mathbf{E} & \operatorname{\Re}\mathbf{E}
		\end{pmatrix}$. Given the regularity of $a^{ij}_1, b^i_1, c_1$, we have  
		$\mathbf{J} \in W^{2,\infty}(0,T; \mathbb{R}^{4m \times 4m})$,  
		$\mathbf{F} \in W^{2,\infty}(0,T; \mathbb{R}^{4m})$.  
		By Carathe\'odory's theorem \cite[Theorem 1.45]{Roubicek},  
		we obtain that  $\mathbf{d}_m \in C^3[0,T]$ with absolutely continuous third derivatives.

		\medskip
		\textbf{Step 2: Uniform estimates.}  
		Multiply \eqref{eq 3.6} by $\overline{d_m^{k\prime}}(t)$ and sum over $k=1,\dots,m$:
		\begin{equation}\label{eq 3.9}
			\left( \mathcal{U}^{\prime\prime}_m, \mathcal{U}^{\prime}_m \right)_{L^2} + B[\mathcal{U}_m, \mathcal{U}^{\prime}_m; t] = \left( F_1(\cdot,t), \mathcal{U}^{\prime}_m \right)_{L^2}.
		\end{equation}
		Noting that  
		$\frac{d}{dt} \| \mathcal{U}^{\prime}_m \|_{L^2}^2 = 2 \Re \left( \mathcal{U}^{\prime\prime}_m, \mathcal{U}^{\prime}_m \right)_{L^2}$  
		and  
		$$\frac{d}{dt} A[\mathcal{U}_m, \mathcal{U}_m; t] = 2 \Re \int_\Omega a^{ij}_1 \mathcal{U}_{m,\mathrm{x}_i} \overline{\mathcal{U}}^{\prime}_{m,\mathrm{x}_j}  d\mathbf{x},$$  
		then take the complex conjugate of \eqref{eq 3.9} and add it to itself, one obtains 
		\begin{equation*}
			\begin{aligned}
				&\quad\frac{d}{dt}\left( \left\| \mathcal{U}_{m}^{\prime}\left( t \right) \right\| _{L^2(\Omega)}^{2}+A\left[ \mathcal{U}_m,\mathcal{U}_m;t \right] \right) 
				\\
				&=2\int_{\Omega}{\Re\left\{ \left(- b^i_1 \mathcal{U}_{m,x_i}-c_1u_m+F_1 \right) \overline{\mathcal{U}}_{m}^{\prime} \right\} d\mathbf{x}}
				\\
				&=C\left( \left\| \mathcal{U}_{m}^{\prime}\left( t \right) \right\| _{L^2(\Omega)}^{2}+\left\| \nabla \mathcal{U}_m\left( t \right) \right\| _{L^2(\Omega)}^{2}+\left\| \mathcal{U}_m\left( t \right) \right\| _{L^2(\Omega)}^{2}+\left\| F_1\left( t \right) \right\| _{L^2(\Omega)}^{2} \right) 
				\\
				&\leqslant C\left( \left\| \mathcal{U}_{m}^{\prime}\left( t \right) \right\| _{L^2}^{2}+A\left[ \mathcal{U}_m,\mathcal{U}_m;t \right] +\left\| F_1\left( t \right) \right\| _{L^2(\Omega)}^{2} \right) 
			\end{aligned}
		\end{equation*}
		a.e. on $[0,T]$. Gronwall's inequality yields:
		\begin{equation*}
			\| \mathcal{U}_{m}^{\prime}(t) \|_{L^2}^2 + A[\mathcal{U}_m, \mathcal{U}_m; t] 
			\leqslant C \left( \| \mathcal{U}_{m}^{\prime}(0) \|_{L^2(\Omega)}^2 + A[\mathcal{U}_m, \mathcal{U}_m; 0] + \int_0^t \| F_1(s) \|_{L^2(\Omega)}^2  ds \right).
		\end{equation*}
		Using \eqref{zero condition} and coercivity $A[\mathcal{U}_m, \mathcal{U}_m; t] \geqslant \theta \| \nabla \mathcal{U}_m \|_{L^2(\Omega)}^2$, we conclude:
		\begin{equation}\label{estimate for u_m}
			\max_{0 \leqslant t \leqslant T} \left( \| \mathcal{U}_{m}^{\prime}(t) \|_{L^2(\Omega)} + \| \mathcal{U}_m(t) \|_{H^1_0(\Omega)} \right) \leqslant C \| F_1 \|_{L^2(0,T;L^2(\Omega))}.
		\end{equation}
		
		Now define $\tilde{\mathcal{U}}_m := \mathcal{U}_{m}^{\prime}$. Differentiating \eqref{eq 3.6} in $t$:
		\begin{equation*}
			\left( \tilde{\mathcal{U}}^{\prime\prime}_m, w_k \right)_{L^2} + B[\tilde{\mathcal{U}}_m, w_k; t] = \left( F_1^{\prime} - b^{i\prime}_1 \mathcal{U}_{m,\mathrm{x}_i} - c^{\prime}_1 \mathcal{U}_m, w_k \right)_{L^2}.
		\end{equation*}
		Following the same argument and using \eqref{estimate for u_m}:
		\begin{equation}\label{estimate 111}
			\max_{0 \leqslant t \leqslant T} \left( \| \tilde{\mathcal{U}}_{m}^{\prime} \|_{L^2(\Omega)} + \| \nabla \tilde{\mathcal{U}}_m \|_{L^2(\Omega)} \right) 
			\leqslant C \| F_1 \|_{H^1(0,T;L^2(\Omega))},
		\end{equation}
		where we used $\|\tilde{\mathcal{U}}_m(0) \|_{H^1_0} = 0$ and
		\begin{equation*}
			\| \tilde{\mathcal{U}}_m^{\prime}(0) \|_{L^2(\Omega)} \leqslant \| F_1(0) \|_{L^2(\Omega)} \leqslant C \| F_1 \|_{H^1(0,T;L^2(\Omega))}.
		\end{equation*}
		
		Similarly, define $\hat{\mathcal{U}}_m := \mathcal{U}_m^{\prime\prime}$. Differentiating \eqref{eq 3.6} twice in $t$:
		\begin{equation*}
			\left( \hat{\mathcal{U}}_m^{\prime \prime}, \hat{\mathcal{U}}_m^{\prime} \right)_{L^2} + B[\hat{\mathcal{U}}_m, \hat{\mathcal{U}}_m^{\prime}; t] = \left( F_1^{\prime\prime} - \mathcal{R}_m, \hat{\mathcal{U}}_m^{\prime} \right)_{L^2}
		\end{equation*}
		with $\mathcal{R}_m := 2b_1^{i\prime} \tilde{\mathcal{U}}_{m,x_i} + 2c_1^{\prime} \tilde{\mathcal{U}}_m + b_1^{i\prime\prime} \mathcal{U}_{m,x_i} + c_1^{\prime\prime} \mathcal{U}_m$. Using \eqref{estimate for u_m} and \eqref{estimate 111}, we obtain
		\begin{equation}\label{estimate of the derivative of u_m}
			\max_{0 \leqslant t \leqslant T} \left( \| \hat{\mathcal{U}}_{m}^{\prime} \|_{L^2(\Omega)} + \|  \hat{\mathcal{U}}_m \|_{H^1_0(\Omega)} \right) 
			\leqslant C \left( \| F_1 \|_{H^2(0,T;L^2(\Omega))} + \| F_1 \|_{H^1(0,T;H^1(\Omega))} \right),
		\end{equation}
		where we note:
		\begin{gather*}
			\| \hat{\mathcal{U}}_m(0) \|_{H^1_0(\Omega)} \leqslant C \| F_1 \|_{H^1(0,T;H^1_0(\Omega))}, \| \hat{\mathcal{U}}_m^{\prime}(0) \|_{L^2(\Omega)} \leqslant C \| F_1 \|_{H^2(0,T;L^2(\Omega))}.
		\end{gather*}
		
		\medskip
		\textbf{Step 3: Existence, uniqueness, and regularity.}  
		By estimate \eqref{estimate for u_m} and standard arguments \cite[Theorems 3–4, pp.407–408]{Evans}, there exists a unique weak solution $\mathcal{U} \in L^\infty(0,T;H^1_0(\Omega))$ with $\mathcal{U}^\prime \in L^\infty(0,T;L^2(\Omega))$ to \eqref{remaining term} satisfying:
		\begin{equation}\label{4.15}
			\underset{0 \leq t \leq T}{\operatorname{ess} \sup}	 \left( \| \mathcal{U}^{\prime}(t) \|_{L^2(\Omega)} + \| \mathcal{U}(t) \|_{H^1_0(\Omega)} \right) \leqslant C \| F_1 \|_{L^2(0,T;L^2(\Omega))}.
		\end{equation}
		From \eqref{estimate of the derivative of u_m} and \eqref{estimate on F_1}, we further have $\mathcal{U}^{\prime \prime} \in L^\infty(0,T;H^1_0(\Omega))$ and 
		\begin{equation}\label{4.16}
			\underset{0 \leq t \leq T}{\operatorname{ess} \sup}
			\left( \| \mathcal{U}^{\prime\prime\prime}(t) \|_{L^2(\Omega)} + \|  \mathcal{U}^{\prime \prime}(t) \|_{H^1_0(\Omega)} \right) 
			\leqslant C \left( \| F_1 \|_{H^2(0,T;L^2(\Omega))} + \| F_1 \|_{H^1(0,T;H^1(\Omega))} \right)\leqslant C\varepsilon.
		\end{equation}
		
		Since $\mathbf{A}_1 - \mathbf{I}$, $\mathbf{b}_1(\cdot,t)$, and $c_1(\cdot,t)$ have compact support in $D$ for all $t \in [0,T]$, the above $\mathcal{U}$ satisfies $-\Delta \mathcal{U} = \partial_{tt}\mathcal{U}$ in $\Omega \backslash \overline{D}$ in the weak sense. By interior elliptic regularity, \eqref{4.15} and \eqref{4.16}, for a.e. $t \in (0,T)$, $\mathcal{U}(t)$ belongs to $H^3_{\mathrm{loc}}(\Omega \backslash \overline{D})$, and for any $\Omega' \Subset \Omega \backslash \overline{D}$:
		\begin{equation*}
			\| \mathcal{U}(\cdot,t) \|_{H^3(\Omega')} \leqslant C \left( \| \partial_{tt} \mathcal{U}(\cdot,t) \|_{H^1(\Omega)} + \| \mathcal{U}(\cdot,t) \|_{H^1(\Omega)} \right) \leqslant C \varepsilon  
		\end{equation*}
		Taking the essential supremum over $t$, we obtain
		\begin{equation*}
			\| \mathcal{U} \|_{L^\infty(0,T;H^3(\Omega'))} \leqslant C \varepsilon.
		\end{equation*}
		By embedding theorem (see \cite{GilbargTrudinger}, Theorem 7.26), foy any $\Omega^\prime \Subset \Omega\backslash \overline{D}$ with Lipschitz boundary, we have
		\begin{equation*}
			\mathcal{U} \in L^\infty(0,T;C^{1,\frac{1}{2}}(\overline{\Omega^\prime}))
		\end{equation*} 
		and 
		\begin{equation*}
			\| \mathcal{U} \|_{L^\infty(0,T;C^{1,\frac{1}{2}}(\overline{\Omega^\prime}))} \leqslant C_3 \varepsilon,
		\end{equation*}
		which implies that properties 1 and 2 hold.
		
		The proof is complete.
	\end{proof}
	
	We now proceed to establish Theorem \ref{thm:main}.
\begin{proof}[Proof of Theorem \ref{1.1}.]
	In the three-dimensional case, we select predetermined parameters $r_0$, $\varepsilon$, and $m$ as follows:
	\begin{equation}\label{4.18}
		\begin{aligned}
			&\hfill r_0 = \min\left\{\,
			\frac{1}{3}(\mathcal{M} + 1)^{-1},\ 
			\frac{1}{6}\min_{1 \le i < j \le n} |\mathbf{x}_i - \mathbf{x}_j|,\ 
			\frac{1}{3}\operatorname{dist}(D, \partial \Omega)
			\,\right\}, \\
			&\hfill \varepsilon = \min\left\{\frac{1}{C_1 + C_2(r_0) + C_3},1\right\}, \\
			&\hfill m = \max\left\{ \left\lfloor 512 r_0^3 \right\rfloor + 1,  M(r_0,\omega) \right\},
		\end{aligned}
	\end{equation}
	where the constants $C_1$, $C_2(r_0)$, and $M(r_0,\omega)$ are specified in Theorem \ref{thm:main_result}, and $C_3$ is given by Theorem \ref{3.1}.  
	
	Next, we choose the initial data $\phi_1,\phi_2$ and boundary data $\psi$ for equation \eqref{1.1} as specified in \eqref{initial and boundary condition}, where the corresponding $u_0$ is determined by the parameters $r_0$, $\varepsilon$, and $m$ in \eqref{4.18}. Let $\mathcal{U}$ denote the unique solution to auxiliary problem \eqref{remaining term}. A direct computation verifies that 
	\begin{equation}\label{eq 4.17}
		u := u_0 + \mathcal{U}
	\end{equation}	
	constitutes the unique solution to problem \eqref{1.1}. By properties 2 of Theorem \ref{remaining term} and the smoothness of $u_0$, it follows that $u \in L^\infty(0,T;C^{1,\frac{1}{2}}(\overline{\Omega^\prime}))$ for any $\Omega^\prime \Subset \Omega \backslash \overline{D}$ with Lipschitz boundary. 
	
	Applying the mean value theorem to \eqref{u0 smalles} and \eqref{3 d} in Theorem \ref{thm:main_result} yields points $\mathbf{z}_i \in B_{3r_0}(\mathbf{x}_i) \backslash \overline{D}$ ($i=1,2,\dots,n$) such that for all $t \in [0,T]$:
	\begin{equation}\label{4.22}
		|\nabla u_0(\mathbf{z}_i,t)| \geqslant \frac{
			\tfrac{1}{16}\sqrt{2m+3} \, r_0^{-3/2} - \left(C_1 + C_2(r_0)\right)\varepsilon
		}{3r_0}.
	\end{equation}
	Combining \eqref{4.18}, \eqref{eq 4.17}, \eqref{4.22}, and properties 2 of Theorem \ref{3.1}, we obtain
	\begin{equation*}
		\begin{aligned}
			|\nabla u(\mathbf{z}_i,t)| &\geqslant \frac{
				\tfrac{1}{16}\sqrt{2m+3} \, r_0^{-3/2} - \left(C_1 + C_2(r_0)\right)\varepsilon
			}{3r_0} - C_3\varepsilon\\
			& >\frac{
				\tfrac{1}{16}\sqrt{1024r_0^3} \, r_0^{-3/2} - 1
			}{3r_0} - 1=\frac{1}{3r_0}-1>\mathcal{M},
		\end{aligned}
	\end{equation*}
	for a.e. $t \in [0,T]$. Here we choose $m \geqslant \left\lfloor 512 r_0^3 \right\rfloor + 1 $ in \eqref{4.18} to guarantee $\tfrac{1}{16}\sqrt{2m+3} \, r_0^{-3/2}\geqslant 2$.
	
	This implies
	\begin{equation*}
		\|\nabla u(\cdot,t)\|_{L^{\infty} (B_{3r_0}(\mathbf{x}_i) \backslash \overline{D})} \geq \mathcal{M}, \quad i=1,2,\dots,n,
	\end{equation*}
	for a.e. $t \in [0,T]$.
	
	Furthermore, under the above construction, the almost-blowup points are confined within the balls $B_{3r_0}(\mathbf{x}_i),i=1,2,...,n,$ which implies that 
	\begin{equation*}
		\operatorname{meas}\big\{ \mathbf{x} \in \bigcup_{i=1}^{n} B_{3r_0}(\mathbf{x}_i)\backslash \overline{D} :|\nabla u(\mathbf{x},t)|  > \mathcal{M} \big\} < 36n\pi r_0^3 < \frac{4n\pi}{3(\mathcal{M} + 1)^3}  \to 0
	\end{equation*}
	as $\mathcal{M} \to \infty$, for a.e. $t \in [0,T]$. The two-dimensional case follows analogously.
	
	The proof is complete.
\end{proof}
	
	\section{Proof of Theorem \ref{thm:main2}}\label{sec:5}

	In this section, we apply identical methodology to the nonlinear second-order hyperbolic equation \eqref{1.2}, proving Theorem~\ref{thm:main2}. Building upon the linear framework, we consider a nonlinear auxiliary problem \eqref{nonlinear hyperbolic equation with source term and zero initial condition} with zero initial conditions, derived from the problem \eqref{1.2}. Using a fixed-point argument in a suitably chosen Banach space, we establish existence and uniqueness of equation \eqref{nonlinear hyperbolic equation with source term and zero initial condition} while demonstrating the solution's H\" older norm remains small. Strategic parameter tuning of $\varepsilon$, $r_0$, and $m$ completes the proof.
	
	Let $\mathscr{U}:\mathbb{R}^d\times [0,T]\rightarrow \mathbb{C}$ satisfy
	\begin{equation}\label{nonlinear hyperbolic equation with source term and zero initial condition}
		\left\{
		\begin{array}{@{}l@{\quad}l}
			\hat{L}_2 \mathscr{U} + \hat{N}(\mathscr{U}, \partial_t \mathscr{U}, \nabla_\mathbf{x} \mathscr{U}) = F_2 & \text{in } \mathbb{R}^d \times (0,T], \\
			\mathscr{U}(\mathbf{x},0) = 0,\ \mathscr{U}_t(\mathbf{x},0) = 0 & \text{in } \mathbb{R}^{d},
		\end{array}
		\right.
	\end{equation}
	where 
	\begin{enumerate}
		\item $
		\hat{L}_2{\mathscr{U}}:=\partial _{tt}^{2}\mathscr{U}-\nabla \cdot (\mathbf{A}_2\nabla \mathscr{U})+\hat{\mathbf{b}}_2\cdot \nabla_{t,\mathbf{x}} \mathscr{U}+\hat{c}_2\mathscr{U}.
		$	
		\item $\hat{\mathbf{b}}_2(\mathbf{x},t)=\left(\sum_{k=2}^{l_0}\left(\mathrm{i}\omega\right)^{k-1}k\beta_{k}u_0^{k-1},\mathbf{b}_2+\sum_{k=2}^{l_0}\sum_{j=1}^{d}k\gamma_{k,j}\left(\partial_{{x}_j}u_0\right)^{k-1}\mathbf{e}_{j}\right),$ here $\mathbf{e}_{j}$ denotes the $d$-dimensional unit vector with its $j$-th component equal to 1.
		\item $
		\nabla_{t,\mathbf{x}} \mathscr{U}:=\left(\partial_t \mathscr{U},\nabla_\mathbf{x} \mathscr{U}\right),	\hat{c}_2(\mathbf{x},t) = c_2 + \sum_{k=2}^{l_0} k \alpha_k u_{0}^{k-1}.
		$
		
		\item The nonlinear term is given by \begin{align*}
			\hat{N}(\mathscr{U},\partial_{t} \mathscr{U},\nabla \mathscr{U}) \nonumber &= \sum_{k=2}^{l_0} \bigg( \alpha_k \sum_{i=2}^k \binom{k}{i} \mathscr{U}^i u_{0}^{k-i} 
			+ \left(\mathrm{i}\omega\right)^{k-i}\beta_k \sum_{i=2}^k \binom{k}{i} (\partial_t \mathscr{U})^i  u_{0}^{k-i} \nonumber \\
			&\quad +  \sum_{j=1}^d \gamma_{k,j}\sum_{i=2}^k \binom{k}{i} (\partial_{{x}_j} \mathscr{U})^i (\partial_{{x}_j} u_{0})^{k-i} \bigg).
		\end{align*}
		
		\item The source term is given by \begin{align*}\label{expression of source term}
			F_2(\mathbf{x},t) &= -\bigg[ \nabla \cdot \left(  \mathbf{I}-\mathbf{A}_2 \right) \nabla u_0 + \mathbf{b}_2 \cdot \nabla u_0 + c_2 u_0 \nonumber \\
			&\quad + \sum_{k=2}^{l_0} \left( (\alpha_k + \left(\mathrm{i}\omega\right)^k \beta_k) u_{0}^k +  \sum_{j=1}^d \gamma_{k,j}(\partial_{{x}_j} u_{0})^k \right) \bigg].
		\end{align*}
	\end{enumerate}
	Next we give some properties for the coefficients and source term in \eqref{nonlinear hyperbolic equation with source term and zero initial condition} which will be used in the proof of Theorem \ref{thm:5.2}.
	\begin{lemma}\label{Lemma 5.1}
		The coefficients and source term in equation \eqref{nonlinear hyperbolic equation with source term and zero initial condition} satisfy the following properties:
		\begin{enumerate}
			\item $\operatorname{supp} \hat{\mathbf{b}}_2(\cdot, t)$, $\operatorname{supp} \hat{c}_2(\cdot, t)$, and $\operatorname{supp} F_2(\cdot, t) \subset D$ for all $t \in [0, T]$.
			\item 
			$
			\|\hat{\mathbf{b}}_2(\cdot, t)\|_{W^{2,\infty}(D)}, \|\hat{c}_2(\cdot, t)\|_{W^{2,\infty}(D)}\leqslant C
			$ for a.e. $t \in [0, T]$.
			Here, constant $C$ depends only on $d$, $l_0$, $D$, $\|\mathbf{b}_2, c_2\|_{L^{\infty}(0,T; W^{2,\infty}(D))}$, and $\| \alpha_k, \beta_k, \gamma_{k,j}\|_{L^{\infty}(0,T; W^{2,\infty}(D))}$.
			\item $
			\|F_2(\cdot, t)\|_{H^{2}(D)} \leqslant C \varepsilon$ 
			for a.e. $t \in [0, T]$. Here the constant $C$ depends on $d$, $l_0$, $D$, $\|\mathbf{A}_2\|_{L^{\infty}(0,T; W^{3,\infty}(D))}$, $\|\mathbf{b}_2, c_2\|_{L^{\infty}(0,T; W^{2,\infty}(D))}$, and $\|\alpha_k, \beta_k, \gamma_{k,j}\|_{L^{\infty}(0,T; W^{2,\infty}(D))}$.
		\end{enumerate}
	\end{lemma}
	
	\begin{proof}
		Properties 1 follows immediately from the compact support in $D$ of $\mathbf{I} - \mathbf{A}_2(\cdot, t)$, $\mathbf{b}_2(\cdot, t)$, $c_2(\cdot, t)$, $\alpha_k(\cdot, t)$, $\beta_k(\cdot, t)$, and $\gamma_{k,j}(\cdot, t)$ for $t \in [0, T]$, $k = 2, \dots, l_0$, $j = 1, \dots, d$, combined with the expressions defining $\hat{\mathbf{b}}_2$, $\hat{c}_2$, and $F_2$.
		
		For properties 2, note that $\alpha_k$, $\beta_k$, and $\gamma_{k,j}$ are smooth, so all their derivatives are bounded in $D$. It remains to show $\partial_{\mathbf{x}}^{\alpha} u_0 \in L^{\infty}(D)$ for all multi-indices $\alpha$ with $|\alpha| \leqslant 3$. By the Sobolev embedding theorem, \eqref{important estimate for H_g}, and \eqref{u_0}, we have
		\begin{equation}\label{5.2}
			\|\partial_{\mathbf{x}}^{\alpha} u_0\|_{L^{\infty}(D)} \leqslant C \|\partial_{\mathbf{x}}^{\alpha} u_0\|_{H^{2}(D)} \leqslant C \|H_g\|_{H^5(D)} \leqslant C \varepsilon \leqslant C, \quad \forall |\alpha| \leqslant 3, \quad \forall t \in [0, T],
		\end{equation}
		where the final inequality holds since $\varepsilon \leqslant 1$ (See Theorem \ref{thm:main_result} and Remark \ref{cf 3.5}). Consequently,
		\begin{equation*}
			\|\hat{\mathbf{b}}_2(\cdot, t)\|_{W^{2,\infty}(D)} \leqslant C \left( \|\mathbf{b}_2(\cdot, t)\|_{W^{2,\infty}(D)} + \sum_{k=2}^{l_0} \|u_0\|_{W^{3,\infty}(D)}^{k-1} \right) \leqslant C \left( 1 + \sum_{k=1}^{l_0-1} \varepsilon^{k} \right) \leqslant C,
		\end{equation*}
		\begin{equation*}
			\|\hat{c}_2(\cdot, t)\|_{W^{2,\infty}(D)}\leqslant  C \left( \|c_2(\cdot, t)\|_{W^{2,\infty}(D)} + \sum_{k=2}^{l_0} \|u_0\|_{W^{2,\infty}(D)}^{k-1} \right) \leqslant C \left( 1 + \sum_{k=1}^{l_0-1} \varepsilon^{k} \right) \leqslant C,
		\end{equation*}
		for a.e. $t \in [0, T]$.
		
		For properties 3, we use the Banach algebra property of $H^2(D)$ and \eqref{5.2}. Since $\varepsilon \leqslant 1$, we have
		\begin{equation}\label{F_2}
			\|F_2(\cdot, t)\|_{H^{2}(D)} \leqslant C \sum_{k=1}^{l_0} \|H_g\|_{H^4(D)}^k \leqslant C \sum_{k=1}^{l_0} \varepsilon^k \leqslant C \varepsilon,
		\end{equation}
		for a.e. $t \in [0, T]$. 
	\end{proof}
	
We now establish the well-posedness of problem \eqref{nonlinear hyperbolic equation with source term and zero initial condition} via Banach's fixed point theorem, combined with an estimate for the nonlinear term $F_2$. This approach yields an explicit characterization of the solution's existence time $T$ in terms of the parameter $\varepsilon$. We show that $T$ admits a uniform positive lower bound independent of $\varepsilon$ and, moreover, becomes arbitrarily large for sufficiently small $\varepsilon$. Furthermore, we prove that the Hölder norm of the solution is bounded by $C \sqrt{\varepsilon}$ for some constant $C > 0$.
	
	\begin{theorem}\label{thm:5.2}
		There exists a time $ T > 0 $ given by \eqref{eq5.17}, depending on $ \varepsilon $, $ d $, $ l_0 $, $ D $, $ \|\mathbf{A}_2\|_{W^{1,\infty}(0,T; W^{3,\infty}(D))} $, $ \|\mathbf{b}_2, c_2\|_{L^{\infty}(0,T; W^{2,\infty}(D))} $, and $\|\alpha_k, \beta_k, \gamma_{k,j}\|_{L^{\infty}(0,T; W^{2,\infty}(D))} $, such that problem \eqref{nonlinear hyperbolic equation with source term and zero initial condition} admits a unique solution $ \mathscr{U} $ satisfying:
		$$
		\mathscr{U} \in L^{\infty}(0, T; H^3(\mathbb{R}^d)) \quad \text{and} \quad \mathscr{U}^{\prime} \in L^{\infty}(0, T; H^2(\mathbb{R}^d))
		$$
		with the following properties:
		\begin{enumerate}
			\item Uniform bounds for $T$: 
			$$
			T\geqslant T_{\mathrm{lower}} 	,
			$$
			where $T_{\mathrm{lower}}$ is given by \eqref{5.17}  and independent of $\varepsilon$.
			\item  H\"older regularity:
			$$
			\mathscr{U} \in L^{\infty}\big([0,T]; C^{1,\frac{1}{2}}(\mathbb{R}^d)\big).
			$$
			\item Smallness estimate:
			$$
			\|\mathscr{U}\|_{L^{\infty}(0,T;C^{1,\frac{1}{2}}(\mathbb{R}^d)) }\leqslant C_6 \sqrt{\varepsilon},
			$$
			where the constant $C_6$ depends only on $d$.
		\end{enumerate}
	\end{theorem}
	\begin{proof}The proof proceeds in four steps.
		
		\medskip
		\textbf{Step 1: Setup of fixed-point framework.} 
		
		Define the Banach space
		$$
		X:=\left\{{u} \in L^{\infty}\left(0, T ; H^1\left(\mathbb{R}^d\right)\right) \mid {u}^{\prime} \in L^{\infty}\left(0, T ; L^2\left(\mathbb{R}^d\right)\right)\right\}
		$$
		with norm
		$$
		\|{u}\|:=\underset{0 \leq t \leq T}{\operatorname{ess} \sup }\left(\|{u}(t)\|_{H^1\left(\mathbb{R}^d\right)}+\left\|{u}^{\prime}(t)\right\|_{L^2\left(\mathbb{R}^d\right)}\right).
		$$
		Introduce the stronger norm 
		$$
		|||{u}|||:=\underset{0 \leq t \leq T}{\operatorname{ess} \sup }\left(\|{u}(t)\|_{H^3\left(\mathbb{R}^d\right)}+\left\|{u}^{\prime}(t)\right\|_{H^{2}\left(\mathbb{R}^d\right)}\right) ,
		$$
		and the closed subset
		$$
		X_\varepsilon:=\left\{{u} \in X \mid\ ||| {u}||| \leq \sqrt{\varepsilon},{u}(0)=0, {u}^{\prime}(0)=0\right\} .
		$$
		For \(\mathscr{V} \in X_\varepsilon\), define \(\mathscr{U} = \mathscr{F}[\mathscr{V}]\) as the solution to the linear initial-value problem
		\begin{equation}\label{linear}
			\left\{
			\begin{aligned}
				&\hat{L}_2 \mathscr{U}= F_2 -  \hat{N}(\mathscr{V}, \nabla \mathscr{V}, \partial_t \mathscr{V}) &&\text{on } \mathbb{R}^d \times (0,T], \\
				&\mathscr{U}(\mathbf{x},0) = 0, \quad \partial_t\mathscr{U}(\mathbf{x},0) = 0 && \text{in } \mathbb{R}^d.
			\end{aligned}
			\right.
		\end{equation}
		By Hahn-Banach Theorem  (see \cite{Sogge}, Theorem 3.2), we know that there exists a unique solution $\mathscr{U}\in L^{\infty}(0,T;H^{3}(\mathbb{R}^d))\cap W^{1,\infty}(0,T;H^{2}(\mathbb{R}^d))$ to problem \eqref{linear}.
		
		\medskip
		\textbf{Step 2: Energy estimate for the problem \eqref{linear}.}
		
		We claim that we have the estimate 
		\begin{equation}\label{key estimate}
			|||\mathscr{U}|||\leqslant C_4(T)T\varepsilon,
		\end{equation}
		where the constant \(C_4(T)\) depends continuously on \(T\), \(d\), \(l_0\), \(D\), \(\|\mathbf{A}_2\|_{W^{1,\infty}(0,T; W^{3,\infty}(D))}\), \(\|\mathbf{b}_2, c_2\|_{L^{\infty}(0,T; W^{2,\infty}(D))}\), and \(\|\alpha_k, \beta_k, \gamma_{k,j}\|_{L^{\infty}(0,T; W^{2,\infty}(D))}\), but not on \(\varepsilon\). Moreover, \(C_4(T)\) is monotonically increasing in \(T\) with \(C_4(0) < \infty\).
		
		To prove \eqref{key estimate},  define $$
		e(t):=\int_{\mathbb{R}^d}|\partial_{t}\mathscr{U}|^2+\nabla \overline{\mathscr{U}} \cdot (\mathbf{A}_2 \nabla \mathscr{U})d\mathbf{x}
		.
		$$ 
		Using Lemma \ref{Lemma 5.1}, we calculate
		\begin{equation}\label{4.9}
			\begin{aligned}
				\frac{d}{d t}e(t) &=  \int_{\mathbb{R}^d} 2\Re\biggl[ \Bigl( F_2 - \hat{\mathbf{b}}_2 \cdot \nabla_{t,\mathbf{x}} \mathscr{U} - \hat{c}_2 \mathscr{U} \\
				&\quad - \hat{N}(\mathscr{V},\partial_{t}\mathscr{V},\nabla_\mathbf{x}\mathscr{V}) \Bigr) \overline{\partial_t \mathscr{U}} \biggr] 
				+ \nabla \overline{\mathscr{U}} \cdot (\partial_{t}\mathbf{A}_2 \nabla \mathscr{U})  \, d\mathbf{x} \\
				&\leqslant C \biggl( e(t) + e(t)^{1/2} \Bigl( \|\mathscr{U}(\cdot,t)\|_{H^1(\mathbb{R}^d)} + \|F_2(\cdot,t)\|_{L^2(\mathbb{R}^d)} \\
				&\quad + \|\hat{N}(\mathscr{V},\partial_{t}\mathscr{V},\nabla_\mathbf{x}\mathscr{V})\|_{L^2(\mathbb{R}^d)} \Bigr) \biggr).
			\end{aligned}
		\end{equation}
		This gives 
		\begin{equation*}
			\frac{d}{d t}e(t)^{\frac{1}{2}}\leqslant C\bigg(e(t)^{\frac{1}{2}}+\|\mathscr{U}(\cdot,t)\|_{H^1(\mathbb{R}^d)} +||F_2(\cdot,t)||_{L^2}+\|\hat{N}(\mathscr{V},\partial_{t}\mathscr{V},\nabla_\mathbf{x}\mathscr{V})\|_{L^2(\mathbb{R}^d)}\bigg)
		\end{equation*}
		which in turn implies that 
		\begin{equation*}
			\begin{aligned}
				&\quad	\frac{d}{d t}\left(e(t)^{\frac{1}{2}}\exp(-Ct)\right)\\
				&\leqslant C\left(\|\mathscr{U}(\cdot,t)\|_{H^1(\mathbb{R}^d)} +||F_2(\cdot,t)||_{L^2(\mathbb{R}^d)}+\|\hat{N}(\mathscr{V},\partial_{t}\mathscr{V},\nabla_\mathbf{x}\mathscr{V})\|_{L^2(\mathbb{R}^d)}\right)\exp(-Ct)\\
				&\leqslant C\left(\|\mathscr{U}(\cdot,t)\|_{H^1(\mathbb{R}^d)} +||F_2(\cdot,t)||_{L^2(\mathbb{R}^d)}+\|\hat{N}(\mathscr{V},\partial_{t}\mathscr{V},\nabla_\mathbf{x}\mathscr{V})\|_{L^2(\mathbb{R}^d)}\right).
			\end{aligned}
		\end{equation*}
		Integrate over $[0,t]$ to obtain
		\begin{equation*}
			e(t)^{\frac{1}{2}}\leqslant \exp\left(Ct\right)C\left(\int_{0}^{t}\|\mathscr{U}(\cdot,\tau)\|_{H^1(\mathbb{R}^d)} +||F_2(\cdot,\tau)||_{L^2(\mathbb{R}^d)}+\|\hat{N}(\mathscr{V},\partial_{t}\mathscr{V},\nabla_\mathbf{x}\mathscr{V})\|_{L^2(\mathbb{R}^d)}d\tau \right).
		\end{equation*}
		To estimate the $L^2$ norm of $\mathscr{U}(\cdot,t)$, we  use the fundamental theorem of calculus and Minkowski's integral inequality to see that 
		\begin{equation*}\label{eq:4.9}
			||\mathscr{U}(\cdot,t)||_{L^2(\mathbb{R}^d)}\leqslant \int_{0}^{t}||\partial_{t}\mathscr{U}(\cdot,\tau)||_{L^2(\mathbb{R}^d)} d\tau.
		\end{equation*}
		Combining the above two inequalities, we obtain 
		\begin{equation*}
			\begin{aligned}
				\|\partial_{t}\mathscr{U}(\cdot,t)\|_{L^2(\mathbb{R}^d)}+\|\mathscr{U}(\cdot,t)\|_{H^1(\mathbb{R}^d)} 
				\leqslant C\left(1+\exp(CT)\right)\int_{0}^{t}  \|\mathscr{U}(\cdot,\tau)\|_{H^1(\mathbb{R}^d)} + \|\partial_{t}\mathscr{U}(\cdot,\tau)\|_{L^2(\mathbb{R}^d)}\\
				+||F_2(\cdot,\tau)||_{L^2}+\|\hat{N}(\mathscr{V},\partial_{t}\mathscr{V},\nabla_\mathbf{x}\mathscr{V})\|_{L^2(\mathbb{R}^d)}d\tau ,
			\end{aligned}
		\end{equation*}
		for a.e. $t \in [0,T]$. Applying Gronwall's inequality to above inequality, we have
		\begin{equation}\label{eq:4.9 }
			\|\partial_{t}\mathscr{U}(\cdot,t)\|_{L^2(\mathbb{R}^d)} + \|\mathscr{U}(\cdot,t)\|_{H^1(\mathbb{R}^d)} \leqslant C_T \int_{0}^{t} ||F_2(\cdot,\tau)||_{L^2}+\|\hat{N}(\mathscr{V},\partial_{t}\mathscr{V},\nabla_\mathbf{x}\mathscr{V})\|_{L^2(\mathbb{R}^d)}d\tau ,
		\end{equation}
		where $C_T=C\left(1+\exp(CT)\right) \exp(C\left(1+\exp(CT)\right))$ .
		
		For any fixed multiindice $\beta$ with $|\beta|\leqslant 2$, 
		$$
		\hat{L}_2\partial_{\mathbf{x}}^{\beta}\mathscr{U}=\partial_{\mathbf{x}}^{\beta}\hat{L}_2\mathscr{U}+[\hat{L}_2,\partial_\mathbf{x}^\beta]\mathscr{U},
		$$
		where $[\hat{L}_2,\partial_\mathbf{x}^\beta] = \hat{L}_2\partial_\mathbf{x}^\beta - \partial_\mathbf{x}^\beta \hat{L}_2$ denotes the commutator, 
		which is a differential operator of order at most 3 in spatial derivatives 
		and at most first order in both space and time variables for mixed derivatives. Repeating for $\partial_\mathbf{x}^{\beta} \mathscr{U}$ the approach used to obtain \eqref{eq:4.9 } yields
		\begin{equation}\label{general estimate for mathscr{U}}
			\begin{aligned}
				&\|\partial_{t}\partial_\mathbf{x}^\beta \mathscr{U}(\cdot,t)\|_{L^2(\mathbb{R}^d)} + \|\partial^{\beta}_{\mathbf{x}}\mathscr{U}(\cdot,t)\|_{H^1(\mathbb{R}^d)}\\
				&\leqslant C_T\bigg( \int_{0}^{t} \|\partial_{\mathbf{x}}^{\beta}\hat{L}_2 \mathscr{U}(\cdot,\tau)\|_{L^2(\mathbb{R}^d)} + \|[\hat{L}_2,\partial_\mathbf{x}^\beta]\mathscr{U}\|_{L^2(\mathbb{R}^d)}\, d\tau \bigg)\\
				&\leqslant  C_T\bigg(\int_{0}^{t} \|\partial_{\mathbf{x}}^{\beta}F_2(\cdot,\tau)\|_{L^2(\mathbb{R}^d)} +\|\partial_{\mathbf{x}}^{\beta}\hat{N}(\mathscr{V},\partial_{t}\mathscr{V},\nabla_\mathbf{x}\mathscr{V})\|_{L^2(\mathbb{R}^d)}\\
				&\quad \quad\quad+\| \mathscr{U}(\cdot,\tau)\|_{H^{3}(\mathbb{R}^d)}+\|\partial_{t} \mathscr{U}(\cdot,t)\|_{H^1(\mathbb{R}^d)}\, d\tau \bigg)\\
				& \leqslant  C_T\bigg(\int_{0}^{t} \|F_2(\cdot,\tau)\|_{H^2(\mathbb{R}^d)}+\|\hat{N}(\mathscr{V},\partial_{t}\mathscr{V},\nabla_\mathbf{x}\mathscr{V})\|_{H^2(\mathbb{R}^d)}\\
				&\quad \quad\quad+\| \mathscr{U}(\cdot,\tau)\|_{H^{3}(\mathbb{R}^d)}+\|\partial_{t} \mathscr{U}(\cdot,\tau)\|_{H^1(\mathbb{R}^d)}\, d\tau \bigg).
			\end{aligned}
		\end{equation}
		Summing over $\beta$ with $|\beta| \leqslant 2$ to $\eqref{general estimate for mathscr{U}}$ and applying Gronwall's inequality, we obtain that
		\begin{equation}\label{5.10}
			||\mathscr{U}(t)||_{H^{3}(\mathbb{R}^d)}+||\mathscr{U}^{\prime}(t)||_{H^{2}(\mathbb{R}^d)} \leqslant C_T\int_{0}^{t} \|F_2(\cdot,\tau)\|_{H^2(\mathbb{R}^d)} +\|\hat{N}(\mathscr{V},\partial_{t}\mathscr{V},\nabla_\mathbf{x}\mathscr{V})\|_{H^2(\mathbb{R}^d)}d\tau ,
		\end{equation}
		for a.e. $t \in [0,T]$. Since \(\mathscr{V} \in X_\varepsilon\) and \(\varepsilon < 1\), Lemma \ref{Lemma 5.1}, the Banach algebra property of \(H^2\), and the explicit structure of \(\hat{N}\) imply that for almost every \(t \in [0,T]\), the following estimate holds:
		\begin{equation}\label{5.11}
			\begin{aligned}
				&\ \quad\|\hat{N}(\mathscr{V},\partial_{t}\mathscr{V},\nabla_\mathbf{x}\mathscr{V})\|_{H^2(\mathbb{R}^d)}\\
				&=\|\hat{N}(\mathscr{V},\partial_{t}\mathscr{V},\nabla_\mathbf{x}\mathscr{V})\|_{H^2(D)}\\
				&\leqslant C\sum_{k=2}^{l_0}\sum_{i=2}^{k}\|\mathscr{V}\|^{i}_{H^2(D)}\|u_0\|^{k-i}_{H^2(D)}+\|\partial_{t}\mathscr{V}\|^{i}_{H^2(D)}\|u_0\|^{k-i}_{H^2(D)}+\|\nabla_{\mathbf{x}}\mathscr{V}\|^{i}_{H^2(D)}\|\nabla u_0\|^{k-i}_{H^2(D)}\\
				&\leqslant C\sum_{k=2}^{l_0}\sum_{i=2}^{k}\varepsilon^{\frac{i}{2}}\varepsilon^{k-i} \leqslant C\varepsilon,
			\end{aligned}
		\end{equation}
		Combining \eqref{5.10}, \eqref{5.11} and properties 3 in Lemma \ref{Lemma 5.1}, \eqref{key estimate} follows.

		\medskip
		\textbf{Step 3: Contraction mapping argument and well-posedness of problem \ref{nonlinear hyperbolic equation with source term and zero initial condition}.}
		
		By properly choosing $T$, we show that 
		\begin{equation}\label{eq5.11}
			\mathscr{F}: X_\varepsilon \to X_\varepsilon
		\end{equation}
		and 
		\begin{equation}\label{eq5.12}
			\|\mathscr{F}[\mathscr{V}_1]-\mathscr{F}[\mathscr{V}_2]\|\leqslant \frac{1}{2}\|\mathscr{V}_1-\mathscr{V}_2\|, \ \forall \mathscr{V}_1,\mathscr{V}_2 \in X_\varepsilon. 
		\end{equation}
		To ensure \eqref{eq5.11}, by \eqref{key estimate}, choose \(T\) such that
		\begin{equation}\label{eq:5.13}
			C_4(T) T \leqslant \frac{1}{\sqrt{\varepsilon}}.
		\end{equation}
		For \eqref{eq5.12}, let \(\mathscr{U}_1 = \mathscr{F}[\mathscr{V}_1]\), \(\mathscr{U}_2 = \mathscr{F}[\mathscr{V}_2]\), and \(\mathscr{W} = \mathscr{U}_1 - \mathscr{U}_2\). Then \(\mathscr{W}\) satisfies
		\begin{equation}
			\left\{
			\begin{aligned}
				&\hat{L}_2\mathscr{W}=  -\big[\hat{N}(\mathscr{V}_1, \nabla \mathscr{V}_1, \partial_t \mathscr{V}_1)-\hat{N}(\mathscr{V}_2, \nabla \mathscr{V}_2, \partial_t \mathscr{V}_2)\big] && \text{in } \mathbb{R}^d \times (0,T], \\
				&\mathscr{W}(\mathbf{x},0) = 0, \quad \partial_t \mathscr{W}(\mathbf{x},0) = 0 && \text{in } \mathbb{R}^d.
			\end{aligned}
			\right.
		\end{equation} By using Lemma \eqref{Lemma 5.1},  similar to \eqref{4.9}, we have 
		\begin{equation}\label{5.15}
			\begin{aligned}
				&~\quad \frac{d}{d t} \int_{\mathbb{R}^d}|\partial_{t}\mathscr{W}|^2+\nabla \overline{\mathscr{W}} \cdot (\mathbf{A}_2\nabla \mathscr{W}) +|\mathscr{W}|^2 d\mathbf{x}\\
				&\leqslant C\int_{\mathbb{R}^d}\left| \partial _t \mathscr{W} \right|^2+|\nabla \mathscr{W}|^2+|\mathscr{W}|^2\\
				&\quad \quad \quad+\left|\hat{N}(\mathscr{V}_1, \nabla \mathscr{V}_1, \partial_t \mathscr{V}_1)-\hat{N}(\mathscr{V}_2, \nabla \mathscr{V}_2, \partial_t \mathscr{V}_2) \right|^2d\mathbf{x}\\
				&\leqslant C\int_{\mathbb{R}^d}|\partial_{t}\mathscr{W}|^2+\nabla \overline{\mathscr{W}} \cdot (\mathbf{A}_2 \nabla \mathscr{W}) +|\mathscr{W}|^2\\
				&\quad \quad \quad +{\varepsilon}\left(\left|\mathscr{V}_1-\mathscr{V}_2 \right|^2+\left|\nabla\mathscr{V}_1-\nabla\mathscr{V}_2 \right|^2+\left|\partial_t\mathscr{V}_1-\partial_t\mathscr{V} \right|^2\right) d\mathbf{x},
			\end{aligned}		
		\end{equation}
		a.e. $t\in \left(0,T\right)$. Here, using the explicit form of $\hat{N}$, estimate \eqref{5.2}, and the fact that $\mathscr{V}_1,\mathscr{V}_2 \in X_{\varepsilon}$ implies $\|\mathscr{V}_i\|_{W^{1,\infty}(\mathbb{R}^d)} ,\|\partial_t\mathscr{V}_i\|_{L^{\infty}(\mathbb{R}^d)}\leqslant C\sqrt{\varepsilon}$ for $i=1,2$, we estimate the difference $\hat{N}(\mathscr{V}_1, \nabla \mathscr{V}_1, \partial_t \mathscr{V}_1)-\hat{N}(\mathscr{V}_2, \nabla \mathscr{V}_2, \partial_t \mathscr{V}_2)$ as follows:
		\begin{equation*}
			\begin{aligned}
				&\int_{\mathbb{R}^d}\bigg|\hat{N}(\mathscr{V}_1, \nabla \mathscr{V}_1, \partial_t \mathscr{V}_1)-\hat{N}(\mathscr{V}_2, \nabla \mathscr{V}_2, \partial_t \mathscr{V}_2) \bigg|^2d\mathbf{x}\\
				&=\int_{\mathbb{R}^d} \bigg|\sum_{k=2}^{l_0} \bigg( \alpha_k \sum_{i=2}^k \binom{k}{i} u_{0}^{k-i} \left(\mathscr{V}_1-\mathscr{V}_2\right)\sum_{\nu =0}^{i-1}\mathscr{V}^{\nu}_{1}\mathscr{V}^{i-1-\nu}_{2}  \\
				&\quad+ \left(\mathrm{i}\omega\right)^{k-i}\beta_k \sum_{i=2}^k \binom{k}{i} u_{0}^{k-i} \left(\partial_{t}\mathscr{V}_1-\partial_{t}\mathscr{V}_2\right)\sum_{\nu =0}^{i-1}(\partial_{t}\mathscr{V}_1)^{\nu}(\partial_{t}\mathscr{V}_2)^{i-1-\nu}  \\
				&\quad +  \sum_{j=1}^d \gamma_{k,j}\sum_{i=2}^k \binom{k}{i} (\partial_{\mathrm{x}_j} u_{0})^{k-i} \left(\partial_{\mathrm{x}_j}\mathscr{V}_1-\partial_{\mathrm{x}_j}\mathscr{V}_2\right)\sum_{\nu =0}^{i-1}(\partial_{\mathrm{x}_j}\mathscr{V}_1)^{\nu}(\partial_{\mathrm{x}_j}\mathscr{V}_2)^{i-1-\nu}  \bigg) \bigg|^2 d\mathbf{x}\\
				&\leqslant C \sum_{k=2}^{l_0}\sum_{i=2}^{k}\varepsilon^{2k-2i}\varepsilon^{i-1}\int_{\mathbb{R}^d} |\mathscr{V}_1-\mathscr{V}_2|^2+|\partial_{t}\mathscr{V}_1-\partial_{t}\mathscr{V}_2|^2+|\nabla_{\mathbf{x}}{\mathscr{V}_1}-\nabla_{\mathbf{x}}{\mathscr{V}_2}|^2\\
				&\leqslant C \varepsilon \int_{\mathbb{R}^d} |\mathscr{V}_1-\mathscr{V}_2|^2+|\partial_{t}\mathscr{V}_1-\partial_{t}\mathscr{V}_2|^2+|\nabla{\mathscr{V}_1}-\nabla{\mathscr{V}_2}|^2,
			\end{aligned}
		\end{equation*}
		for a.e. $t \in [0,T]$. Invoking Gronwall's inequality to \eqref{5.15}, and using the definition of \(\|\cdot\|\),  we deduce that 
		% The constant  $C$ depends on $||\mathscr{V},\hat{\mathscr{V}},\nabla\mathscr{V},\nabla\hat{\mathscr{V}},\partial_t\mathscr{V},\partial_t\hat{\mathscr{V}}||_{L^\infty(\mathbb{R}^N)}$, $\theta$ and the norm of $\mathbf{A}_2, \mathbf{b}_2, c_2,\alpha_k,\beta_{k},\gamma_{k,j}$.  This quantity is bounded since the functions $\mathscr{V},\hat{\mathscr{V}}\in X_{\varepsilon}$ and Sobolev embedding theorem. 
		\begin{equation}\label{5.16}
			\begin{aligned}
				\underset{0 \leq t \leq T}{\operatorname{ess} \sup }&\int_{\mathbb{R}^d} \left| \partial _t \mathscr{W} \right|^2+|\nabla \mathscr{W}|^2+|\mathscr{W}|^2 d\mathbf{x}\\
				&\leqslant C_5(T){\varepsilon}\int_{0}^{T}\int_{\mathbb{R}^d} \left|\mathscr{V}-\hat{\mathscr{V}} \right|^2+\left|\nabla\mathscr{V}-\nabla\hat{\mathscr{V}} \right|^2+\left|\partial_t\mathscr{V}-\partial_t\hat{\mathscr{V}} \right|^2d\mathbf{x} dt\\
				&\leqslant C_5(T){\varepsilon}T||\mathscr{V}-\hat{\mathscr{V}}||^2,
			\end{aligned}
		\end{equation}
		where \(C_5(T)\) has the same dependencies as \(C_4(T)\) and is monotonically increasing in \(T\) with \(C_5(0) < \infty\). 	To ensure \(\|\mathscr{W}\|=\| \mathscr{F}[\mathscr{V}_1]-\mathscr{F}[\mathscr{V}_2]\| \leqslant \frac{1}{2} \|\mathscr{V}_1 - \mathscr{V}_2\|\), we require by \eqref{5.16}, 
		\begin{equation}\label{eq C_T(5)}
			C_5(T) T \leqslant \frac{1}{4{\varepsilon}}.
		\end{equation}
		Set $T=T_{\text{lifespan}}$, where
		\begin{equation}\label{eq5.17}
			T_{\mathrm{lifespan}}:=\max\{T|	C_4(T) T \leqslant \frac{1}{\sqrt{\varepsilon}}\ \text{and}\ C_5(T) T \leqslant \frac{1}{4{\varepsilon}}\}.
		\end{equation}
		Then \eqref{eq5.11} and \eqref{eq5.12} hold simultaneously.
		
		Now we are ready to prove the existence and uniqueness of problem \eqref{nonlinear hyperbolic equation with source term and zero initial condition}. Select any $\mathscr{U}_0 \in X_\varepsilon$, according to  Banach's fixed point  Theorem, if we inductively define $\mathscr{U}_{k+1}:=\mathscr{F}\left[\mathscr{U}_k\right]$ for $k=0, \ldots$, then ${\mathscr{U}}_k \rightarrow {\mathscr{U}}$ in $X_\varepsilon$ and
		$$
		\mathscr{F}[\mathscr{U}]=\mathscr{U}.
		$$
		Furthermore, since $|||\mathscr{U}_k|||\leq \sqrt{\varepsilon}$, we have $\mathscr{U} \in X_\varepsilon$. Uniqueness follows from \eqref{eq5.12}. By Sobolev embedding theorem, properties 1 and 2 hold.

		\medskip
		\textbf{Step 4: Uniform lower bounds for \(T\) of problem \eqref{nonlinear hyperbolic equation with source term and zero initial condition}.}
		Define	
		\begin{equation}\label{5.17}
			T_{\mathrm{lower}}:=\max\{T|	C_4(T) T \leqslant 1\ \text{and}\ C_5(T) T \leqslant \frac{1}{4}\}.
		\end{equation}
		Since \(\varepsilon < 1\), we have \(\frac{1}{\sqrt{\varepsilon}}>1\) and \(\frac{1}{4 \varepsilon} > \frac{1}{4}\). As \(C_4(T)T\) and \(C_5(T)T\) are continuous and increasing, it follows that \( T_{\mathrm{lifespan}}\geqslant T_{\mathrm{lower}} \), with \(T_{\mathrm{lower}}\) independent of \(\varepsilon\) which implies that properties 3 holds.
		
		The proof is complete.
	\end{proof}

	\begin{remark}\label{remark 5.3}
		The argument in the proof of Theorem~\ref{thm:5.2} follows the methodology of Theorem 3 in Chapter 12.2 of \cite{Evans}. To guarantee that the mapping $\mathscr{F}$ preserves the space $X_{\varepsilon}$ and satisfies the contraction property, the time $T$ is chosen to satisfy conditions \eqref{eq:5.13} and \eqref{eq C_T(5)}. Notably, as $\varepsilon$ decreases, the existence time $T$ for problem \eqref{nonlinear hyperbolic equation with source term and zero initial condition} increases—thereby extending the scope of Theorem 3 in Chapter 12.2 of \cite{Evans}. This improvement arises from three key structural features: (i) the problem \eqref{nonlinear hyperbolic equation with source term and zero initial condition} has trivial initial data; (ii) the polynomial nonlinearity $\hat{N}$ has compact support in $D$ and an explicit form; and (iii) the source term $F_2$ has an explicit expression and its $H^2(D)$-norm is small in terms of $\varepsilon$. Together with the careful construction of the closed set $X_\varepsilon$ and the properties of the function $u_0$, these features enable the existence of a solution to problem \eqref{nonlinear hyperbolic equation with source term and zero initial condition} on arbitrarily large time intervals $T$.
	\end{remark}
	
We now proceed to prove Theorem \ref{1.2}. As the argument closely parallels that of Theorem \ref{1.1}, we provide the complete details below for clarity and self-containedness.
\begin{proof}[Proof of Theorem \ref{1.2}.]\label{proof 1.2}
	In the three-dimensional case, we select the predetermined parameters $r_0$, $\varepsilon$, and $m$ as follows:
	\begin{equation}\label{eeq 5.20}
		\begin{aligned}
			&\hfill r_0 = \min\left\{\frac{1}{3}(\mathcal{M} + 1)^{-1},\ \frac{1}{6}\min_{1 \leq i < j \leq n} |\mathbf{x}_i - \mathbf{x}_j|\right\}, \\
			&\hfill \varepsilon = \min\left\{\frac{1}{C_1 + (C_6)^2},\frac{1}{\left(C_4(T)T\right)^2},\frac{1}{4C_5(T)T},1\right\}, \\
			&\hfill m = \max\left\{ \left\lfloor \frac{256(2+C_2(r_0)\varepsilon)^2r_{0}^{3}-3}{2} \right\rfloor + 1,  M(r_0,\omega) \right\},
		\end{aligned}
	\end{equation}
	where the constants $C_1$, $C_2(r_0)$, and $M(r_0,\omega)$ are specified in Theorem~\ref{thm:main_result}; $C_4(T)$ and $C_5(T)$ are given by \eqref{key estimate} and \eqref{5.16}, respectively, in Theorem~\ref{thm:5.2}; $C_6$ is given in Theorem~\ref{thm:5.2}; and the prescribed time $T$ is as given in Theorem~\ref{thm:main2}.
	
	Then, we choose the initial inputs for problem \eqref{1.2} as specified in \eqref{initial condition}, with the corresponding $u_0$ determined by the parameters $r_0$, $\varepsilon$, and $m$ in \eqref{eeq 5.20}. Let $\mathscr{U}$ denote the unique solution to auxiliary problem \eqref{nonlinear hyperbolic equation with source term and zero initial condition}. A direct computation verifies that
	\begin{equation}\label{eq 5.21}
		U := u_0 + \mathscr{U}
	\end{equation}
	constitutes the unique solution to problem \eqref{1.2}. Since $u_0$ is smooth and defined on $[0,\infty)$, the lifespan of $U$ is entirely determined by that of $\mathscr{U}$. As established by the lifespan of $\mathscr{U}$ in \eqref{eq5.17} and the explicit expression for $\varepsilon$ in \eqref{eeq 5.20}, $U$ exists on the interval $[0,T]$, where $T$ denotes the prespecified time interval given in Theorem \ref{thm:main2}. Moreover, Theorems \ref{thm:5.2} and \ref{thm:main_result} together with Remark \ref{remark 2.3} yield $U \in L^{\infty}\big(0,T;H^{3}_{\text{loc}}(\mathbb{R}^d)\big)$ and $U_t \in L^{\infty}\big(0,T;H^{2}_{\text{loc}}(\mathbb{R}^d)\big)$.
	
	Next, applying the mean value theorem to \eqref{u0 smalles} and \eqref{3 d} in Theorem \ref{3.1} yields points $\mathbf{z}_i \in B_{3r_0}(\mathbf{x}_i) \backslash \overline{D}$ ($i=1,2,\dots,n$) such that for all $t \in [0,T]$:
	\begin{equation}\label{5.23}
		|\nabla u_0(\mathbf{z}_i,t)| \geqslant \frac{
			\tfrac{1}{16}\sqrt{2m+3} \, r_0^{-3/2} - \left(C_1 + C_2(r_0)\right)\varepsilon
		}{3r_0}.
	\end{equation}
	Combining \eqref{eeq 5.20}, \eqref{eq 5.21}, \eqref{5.23}, and properties 3 in Theorem \ref{thm:5.2}, we obtain
	\begin{equation*}
		\begin{aligned}
			|\nabla U(\mathbf{z}_i,t)| &\geqslant \frac{\left(
				\tfrac{1}{16}\sqrt{2m+3} \, r_0^{-3/2} -C_2(r_0)\varepsilon\right)-  C_1\varepsilon
			}{3r_0} - C_6\sqrt{\varepsilon} \\
			& > \frac{2 - 1}{3r_0} - 1 = \frac{1}{3r_0} - 1 > \mathcal{M},
		\end{aligned}
	\end{equation*}
	for a.e. $t \in [0,T]$. Here, we choose $m \geqslant \left\lfloor \frac{256(2 + C_2(r_0)\varepsilon)^2 r_{0}^{3} - 3}{2} \right\rfloor +1$ in \eqref{eeq 5.20} to guarantee that  $	\tfrac{1}{16}\sqrt{2m+3} \, r_0^{-3/2} -C_2(r_0)\varepsilon>2 $. 
	
	This implies
	\begin{equation*}
		\|\nabla U(\cdot,t)\|_{L^{\infty} (B_{3r_0}(\mathbf{x}_i) \backslash \overline{D})} \geq \mathcal{M}, \quad i=1,2,\dots,n,
	\end{equation*}
	for a.e. $t\in [0,T]$.
	
	Furthermore, under the above construction, the almost-blowup points are confined within the balls $B_{3r_0}(\mathbf{x}_i)$ which implies that 
	\begin{equation*}
		\operatorname{meas}\left\{ \mathbf{x} \in \bigcup_{i=1}^{n} \left(B_{3r_0}(\mathbf{x}_i)\backslash \overline{D}\right) : |\nabla U(\mathbf{x},t)| > \mathcal{M} \right\} < 36n\pi r_0^3 < \frac{4n\pi}{3(\mathcal{M} + 1)^3}  \to 0
	\end{equation*}
	as $\mathcal{M} \to \infty$, for almost every $t \in [0,T]$. The two-dimensional case follows analogously.
	
	The proof is complete.
\end{proof}

	\section*{Acknowledgment}
	The work of H. Diao is is supported by National Natural Science Foundation of China  (No. 12371422), the Fundamental Research Funds for the Central Universities, JLU.   The work of H. Liu is supported by the Hong Kong RGC General Research Funds (projects 11304224, 11311122 and 11303125),  the NSFC/RGC Joint Research Fund (project N\_CityU101/21), the France-Hong Kong ANR/RGC Joint Research Grant, A-CityU203/19.

	\medskip
	
	%\Blu{\textbf{Data availability statement.} Data sharing is not applicable to this article as no datasets were generated or analysed during the current study.}


\begin{thebibliography}{99}
	
	\bibitem{Abramowitz} 
	M.~Abramowitz and I.~A.~Stegun, 
	\textit{Handbook of Mathematical Functions: With Formulas, Graphs, and Mathematical Tables}, 
	Dover Publications, New York, 1965.
	
	\bibitem{alinhac1995blowup}
	S.~Alinhac,
	\textit{Blowup for Nonlinear Hyperbolic Equations},
	Progress in Mathematics, Vol. 135,
	Birkhäuser Verlag, Basel, 1995.
	
	\bibitem{alinhac1999}
	S.~Alinhac, 
	Blow-up of small data solutions for a quasilinear wave equation in two space dimensions, 
	\textit{Annals of Mathematics}, 
	\textbf{149}(1) (1999), 1--28.
	
	\bibitem{Antonini} 
	C.~A.~Antonini, 
	Smooth approximation of Lipschitz domains, weak curvatures and isocapacitary estimates, 
	\textit{Calculus of Variations and Partial Differential Equations}, 
	\textbf{63}(4) (2024), Paper No. 91, 34 pp.
	
	\bibitem{ColtonKress} 
	D.~L.~Colton and R.~Kress, 
	\textit{Inverse Acoustic and Electromagnetic Scattering Theory}, 
	4th ed., 
	Vol. 93 of Applied Mathematical Sciences, 
	Springer, 2019.
	
	\bibitem{collot2018typeII}
	C.~Collot, 
	Type II blow-up manifolds for the energy supercritical semilinear wave equation, 
	\textit{Memoirs of the American Mathematical Society}, 
	\textbf{252}(1205) (2018).
	
	\bibitem{craster2012acoustic}
	R.~V.~Craster and S.~Guenneau (eds.), 
	\textit{Acoustic Metamaterials: Negative Refraction, Imaging, Lensing and Cloaking}, 
	Springer, Dordrecht, 2012.
	
	\bibitem{DonningerRao2020Blowup}
	R.~Donninger and Z.~Rao, 
	Blowup stability at optimal regularity for the critical wave equation,
	\textit{Advances in Mathematics}, 
	\textbf{370} (2020), 107219.
	
	\bibitem{DonningerSchorkhuber2016}
	R.~Donninger and B.~Sch\"orkhuber,
	On blowup in supercritical wave equations,
	\textit{Communications in Mathematical Physics}, 
	\textbf{346} (2016), 907--943.
	
	\bibitem{duyckaerts2011universality}
	T.~Duyckaerts, C.~E.~Kenig and F.~Merle, 
	Universality of blow-up profile for small radial type II blow-up solutions of the energy-critical wave equation, 
	\textit{Journal of the European Mathematical Society}, 
	\textbf{13}(3) (2011), 533--599.
	
	\bibitem{duyckaerts2013classification}
	T.~Duyckaerts, C.~E.~Kenig and F.~Merle, 
	Classification of radial solutions of the focusing, energy-critical wave equation, 
	\textit{Cambridge Journal of Mathematics}, 
	\textbf{1}(1) (2013), 75--144.
	
	\bibitem{Evans} 
	L.~C.~Evans, 
	\textit{Partial Differential Equations}, 
	2nd ed., Vol. 19 of Graduate Studies in Mathematics, 
	American Mathematical Society, Providence, RI, 2017.
	
	\bibitem{fichtner2011full}
	A.~Fichtner, 
	\textit{Full Seismic Waveform Modelling and Inversion}, 
	Springer, Berlin/Heidelberg, 2011.
	
	\bibitem{GilbargTrudinger} 
	D.~Gilbarg and N.~S.~Trudinger, 
	\textit{Elliptic Partial Differential Equations of Second Order}, 
	Springer-Verlag, Berlin, 1977.
	
	\bibitem{glassey1973blowup}
	R.~T.~Glassey, 
	Blow-up theorems for nonlinear wave equations, 
	\textit{Mathematische Zeitschrift}, 
	\textbf{132} (1973), 183--203.
	
	\bibitem{Hormander} 
	L.~H\"ormander, 
	\textit{The Analysis of Linear Partial Differential Operators III: Pseudo-Differential Operators}, 
	Springer, 1985.
	
	\bibitem{hormander1997lectures}
	L.~H\"ormander, 
	\textit{Lectures on Nonlinear Hyperbolic Differential Equations}, 
	Math\'ematiques \& Applications, Vol. 26, 
	Springer-Verlag, Berlin/Heidelberg, 1997.
	
	%\bibitem{HuLWZ} 
	%Y.~Hu, H.~Liu, X.~Wang and D.~Zhang, 
	%Generating Customized Field Concentration via Virtual Surface Transmission Resonance, 
	%\textit{SIAM Journal on Applied Mathematics}, 
	%\textbf{85}(3) (2025), 1143--1171.
	
	\bibitem{joannopoulos2008photonic}
	J.~D.~Joannopoulos, S.~G.~Johnson, J.~N.~Winn and R.~D.~Meade, 
	\textit{Photonic Crystals: Molding the Flow of Light}, 
	2nd ed., 
	Princeton University Press, 2008.
	
	\bibitem{John} 
	F.~John, 
	Blow-up of solutions of nonlinear wave equations in three space dimensions, 
	\textit{Manuscripta Mathematica}, 
	\textbf{28}(1-3) (1979), 235--268.
	
	\bibitem{kato1980}
	T.~Kato, 
	Blow-up of solutions of some nonlinear hyperbolic equations, 
	\textit{Communications on Pure and Applied Mathematics}, 
	\textbf{33}(4) (1980), 501--505.
	
	\bibitem{kenig2008global}
	C.~E.~Kenig and F.~Merle, 
	Global well-posedness, scattering and blow-up for the energy-critical focusing non-linear wave equation, 
	\textit{Acta Mathematica}, 
	\textbf{201} (2008), 147--212.
	
	\bibitem{krieger2009slow}
	J.~Krieger, W.~Schlag and D.~Tataru, 
	Slow blow-up solutions for the $H^1(\mathbb{R}^3)$ critical focusing semilinear wave equation, 
	\textit{Duke Mathematical Journal}, 
	\textbf{147}(1) (2009), 1--53.
	
	\bibitem{Krieger} 
	J.~Krieger, W.~Schlag and D.~Tataru, 
	Renormalization and blow up for charge one equivariant critical wave maps, 
	\textit{Inventiones Mathematicae}, 
	\textbf{171}(3) (2009), 543--615.
	
	\bibitem{Lasiecka} 
	I.~Lasiecka, J.~L.~Lions and R.~Triggiani, 
	Nonhomogeneous boundary value problems for second order hyperbolic operators, 
	\textit{Journal de Math\'ematiques Pures et Appliqu\'ees}, 
	\textbf{65}(2) (1986), 149--192.
	
	\bibitem{levine1974}
	H.~A.~Levine,
	Instability and nonexistence of global solutions to nonlinear wave equations of the form $Pu_{tt} = -Au + \mathcal{F}(u)$,
	\textit{Transactions of the American Mathematical Society},
	\textbf{192} (1974), 1--21.
	
	\bibitem{LiuZou} 
	H.~Liu and J.~Zou, 
	Zeros of the Bessel and spherical Bessel functions and their applications for uniqueness in inverse acoustic obstacle scattering, 
	\textit{IMA Journal of Applied Mathematics}, 
	\textbf{72}(6) (2007), 817--831.
	
	\bibitem{Lohofer} 
	G.~Loh\"ormander, 
	Inequalities for the associated Legendre functions, 
	\textit{Journal of Approximation Theory}, 
	\textbf{95}(2) (1998), 178--193.
	
	\bibitem{MerleZaag} 
	F.~Merle and H.~Zaag, 
	Determination of the blow-up rate for the semilinear wave equation, 
	\textit{American Journal of Mathematics}, 
	\textbf{127}(5) (2005), 1073--1118.
	
	\bibitem{rodnianski2010formation}
	I.~Rodnianski and J.~Sterbenz,
	On the formation of singularities in the critical $O(3)$ $\sigma$-model,
	\textit{Annals of Mathematics},
	\textbf{172}(1) (2010), 187--242.
	
	\bibitem{Roubicek} 
	T.~Roubíček, 
	\textit{Nonlinear Partial Differential Equations with Applications}, 
	Vol. 153 of International Series of Numerical Mathematics, 
	Birkhäuser, Basel, 2013.
	
	\bibitem{Sideris} 
	T.~C.~Sideris, 
	Nonexistence of global solutions to semilinear wave equations in high dimensions, 
	\textit{Journal of Differential Equations}, 
	\textbf{52}(3) (1984), 378--406.
	
	\bibitem{Sogge} 
	C.~D.~Sogge, 
	\textit{Lectures on Non-Linear Wave Equations}, 
	International Press, Boston, MA, 1995.
	
	\bibitem{Strauss}
	W.~A.~Strauss, 
	\textit{Nonlinear Wave Equations}, 
	CBMS Regional Conference Series in Mathematics, Vol. 73, 
	American Mathematical Society, 1989.
	
	\bibitem{Weck} 
	N.~Weck, 
	Approximation by Herglotz wave functions, 
	\textit{Mathematical Methods in the Applied Sciences}, 
	\textbf{27}(2) (2004), 155--162.
	
\end{thebibliography}
\end{document}